\documentclass[11pt]{amsart}

\usepackage{amsfonts}
\usepackage{amssymb}
\usepackage{amsthm}
\usepackage{amsmath}
\usepackage{amscd}
\usepackage{hyperref}
\usepackage{enumitem}
\usepackage{tikz}
\usepackage{pgf}
\usepackage{graphicx}
\usepackage[top=3.5cm, bottom=3.5cm, left=2.5cm, right=2.5cm]{geometry}

\title{Scaling limits of Cayley graphs with polynomially growing balls}
\date{}
\author{Romain Tessera and Matthew C. H. Tointon}
\address{Laboratoire de Math\'ematiques d'Orsay, Univ.~Paris-Sud, CNRS, Universit\'e Paris-Saclay, 91405 Orsay, France}
\email{tessera@phare.normalesup.org}
\address{Insitut de Math\'ematiques, Universit\'e de Neuch\^atel, Rue Emile-Argand 11, CH-2000 Neuch\^atel, Switzerland}
\email{matthew.tointon@unine.ch}
\thanks{The first author is supported by grant ANR-14-CE25-0004 ``GAMME". The second author is supported by grant FN 200021\_163417/1 of the Swiss National Fund for scientific research; for earlier periods of this project he was supported by ERC grant GA617129 `GeTeMo' and a Junior Research Fellowship from Homerton College, University of Cambridge.}
\newtheorem{prop}{Proposition}[section]
\newtheorem{theorem}[prop]{Theorem}
\newtheorem{lemma}[prop]{Lemma}
\newtheorem{corollary}[prop]{Corollary}

\newtheorem*{thmm}{Theorem}

\theoremstyle{definition}
\newtheorem{definition}[prop]{Definition}

\theoremstyle{remark}

\newtheorem*{remark*}{Remark}
\newtheorem*{remarks*}{Remarks}
\newtheorem{remark}[prop]{Remark}

\theoremstyle{theorem}
\newcommand{\hdim}{\text{\textup{hdim}}\,}
\newcommand{\PP}{\mathcal{P}}
\newcommand{\R}{\mathbb{R}}

\newcommand{\CC}{\mathfrak{C}}

\newcommand{\N}{\mathbb{N}}
\newcommand{\Z}{\mathbb{Z}}
\newcommand{\Q}{\mathbb{Q}}

\newcommand{\HHH}{\mathbb{H}}

\newcommand{\n}{\mathfrak{n}}
\newcommand{\g}{\mathfrak{g}}

\newcommand{\m}{\mathfrak{m}}

\newcommand{\eps}{\varepsilon}
\newcommand{\ord}{\text{\textup{ord}}}

\newcommand{\diam}{\text{\textup{diam}}}

\newcommand*{\vol}{\mathop{\textup{vol}}\nolimits}

\numberwithin{equation}{section}

\begin{document}
\maketitle
\begin{abstract}
Benjamini, Finucane and the first author have shown that if $(\Gamma_n,S_n)$ is a sequence of Cayley graphs such that $|S_n^n|\ll n^D|S_n|$, then the sequence $(\Gamma_{n},\frac{d_{S_{n}}}{n})$ is relatively compact for the Gromov--Hausdorff topology and every cluster point is a connected nilpotent Lie group equipped with a left-invariant sub-Finsler metric. In this paper we show that the dimension of such a cluster point is bounded by $D$, and that, under the stronger bound $|S_n^n|\ll n^D$, the homogeneous dimension of a cluster point is bounded by $D$. Our approach is roughly to use a well-known structure theorem for approximate groups due to Breuillard, Green and Tao to replace $S_n^n$ with a coset nilprogression of bounded rank, and then to use results about nilprogressions from a previous paper of ours to study the ultralimits of such coset nilprogressions. As an application we bound the dimension of the scaling limit of a sequence of vertex-transitive graphs of large diameter. We also recover and effectivise parts of an argument of Tao concerning the further growth of single set $S$ satisfying the bound $|S^n|\le Mn^D|S|$.
\end{abstract}

\tableofcontents

\section{Introduction}

%\subsection{Gromov's Theorem and its quantitative versions}
We start by recalling Gromov's famous polynomial growth theorem.

\begin{thmm}[\cite{gromov,vdD-W}]
Let $\Gamma$ be a group generated by a finite symmetric subset $S$. If there exists an unbounded sequence of integers $m_n$, and constants $C$ and $D$ such that
$|S^{m_n}|\leq Cm_n^D$ for all $n$, then $\Gamma$ is virtually torsion-free nilpotent.
\end{thmm}

Conversely, we know from Bass and Guivarc'h \cite{bass,Gui} that for every torsion-free nilpotent group $\Gamma$ equipped with a symmetric finite generating subset $S$ there exists a constant $C=C(\Gamma,S)\geq 1$ such that for every $n\in \N$ we have
$$C^{-1}n^{\hdim(\Gamma)}\leq |S^n|\leq Cn^{\hdim(\Gamma)},$$
where $\hdim(\Gamma)$ is the homogeneous dimension of $\Gamma$ (it is in particular an integer). 

More recently, Shalom and Tao \cite{st} showed that for every $D$ there exists $n_0$ such that for every group $\Gamma$ generated by a finite symmetric subset $S$, if there exists just one integer $n_1\geq n_0$ such that
\begin{equation}\label{eq:abs.poly.growth}
|S^{n_1}|\leq n_1^D
\end{equation}
then $\Gamma$ is virtually nilpotent. Breuillard, Green and Tao \cite[Corollary 11.5]{bgt} subsequently showed that \eqref{eq:abs.poly.growth} could be weakened to a ``relative'' polynomial-growth condition of the form
\begin{equation}\label{hom.poly.growth}
|S^{n_1}|\leq n_1^D|S|.
\end{equation}

%then $\Gamma$ has a finite index subgroup $\Gamma'$ with a finite normal subgroup $H$ such that $\Gamma'/H$ is nilpotent. In quantitative versions of this statement, the index of $\Gamma'$ and rank and step of $\Gamma'/H$ are bounded in terms of $C$ and $D$, while the diameter of $H$ is shown to be ``small" in comparison to $n_1$ (note that this gives a non-trivial statement even when $\Gamma$ is finite).

One can actually go further and show that if one takes a ``large'' group $\Gamma$ satisfying \eqref{hom.poly.growth} then when  ``zooming out'' at scale $n_1$ it ``looks like'' a connected nilpotent Lie group. To make this idea precise, one can use the \emph{Gromov--Hausdorff topology}, or \emph{GH topology}, which gives a way to measure how ``close'' two metric spaces are to one another; see \S\ref{sec:GH-limit} for a precise definition. To study $\Gamma$ ``zoomed out'' at scale $n_1$, one can consider the rescaled Cayley graph $(\Gamma,\frac{d_{S}}{n_1})$ with respect to the GH topology.

The first author obtained the following result in this direction in joint work with Benjamini and Finucane.
\begin{theorem}[{\cite[Theorem 3.2.2.]{bft}}]\label{thm:bft}
Let $(\Gamma_n,S_n)$ be a sequence of Cayley graphs such that $|S_n^n|\ll n^D|S_n|$. Then for every sequence $m_n\gg n$ the sequence $(\Gamma_{n},\frac{d_{S_{n}}}{m_n})$ is GH-relatively compact, and every cluster point is a connected nilpotent Lie group equipped with a left-invariant sub-Finsler metric. \end{theorem}
\begin{remarks*}
The authors of \cite{bft} use the terminology \emph{Carnot--Carath\'eodory metric} to mean a sub-Finsler metric (we caution that sometimes Carnot--Carath\'eodory is used more specifically to mean sub-Riemmannian). Moreover, whilst \cite[Theorem 3.2.2.]{bft} is stated for finite graphs this is in fact a misprint and the proof is also valid for infinite graphs. Finally, \cite[Theorem 3.2.2.]{bft} has the stronger hypothesis that $|S_n^n|=O(n^D)$, but the same proof yields the same result under the hypothesis of Theorem \ref{thm:bft} that $|S_n^n|=O(n^D|S_n|)$.
\end{remarks*}

\bigskip
\noindent\textsc{Main new results.} Our first result in this paper bounds by $D$ the dimension of any cluster point arising from Theorem \ref{thm:bft}. 
\begin{theorem}\label{thm:relative}
Let $(\Gamma_n,S_n)$ be a sequence of Cayley graphs such that $|S_n^n|\ll n^D|S_n|$.  Then for every sequence $m_n\gg n$ the sequence $(\Gamma_{n},\frac{d_{S_{n}}}{m_n})$ is GH-relatively compact, and every cluster point is a connected nilpotent Lie group of dimension at most $D$ equipped with a left-invariant sub-Finsler metric.
\end{theorem}
\begin{remark}\label{rem:hom.dim}
In particular, the homogeneous dimension of every cluster point coming from Theorem \ref{thm:relative} is at most $(D-1)D/2+1$.
\end{remark}

In light of the Bass--Guivarc'h formula, it is tempting to wonder whether even the homogeneous dimension should be bounded by $D$. However, the following example shows that Theorem \ref{thm:relative} and Remark \ref{rem:hom.dim} are both sharp. Consider the subset $S$ of the discrete Heisenberg group $\HHH(\Z)=\left(\begin{smallmatrix}1 & \Z & \Z\\0&1&\Z\\0&0&1\end{smallmatrix}\right)$ defined by
\[
S=\left(\begin{smallmatrix}1 & [-n,n] & [-n^3,n^3]\\0&1&[-n,n]\\0&0&1\end{smallmatrix}\right),
\]
which also appears in \cite[Example 1.11]{tao}. This set $S_n$ satisfies $|S_n^n|\le n^3|S_n|$, while one easily checks that the scaling limit of $(\HHH(\Z),d_{S_n}/n)$ is the real Heisenberg group, whose homogeneous dimension is $4$ but whose dimension is indeed $3$ (note that, although $S$ is not symmetric, this can be fixed by considering the set $S\cup S^{-1}$ in its place; we leave the details to the reader).

Nonetheless, replacing the relative growth hypothesis \eqref{hom.poly.growth} with the stronger absolute growth hypothesis \eqref{eq:abs.poly.growth} we can indeed bound the homogeneous dimension by $D$, as follows.

\begin{theorem}\label{thm:conj2}
Let $(\Gamma_n,S_n)$ be a sequence of Cayley graphs such that $|S_n^n|\ll n^D$.  Then for every sequence $m_n\gg n$ the sequence $(\Gamma_{n},\frac{d_{S_{n}}}{m_n})$ is GH-relatively compact, and every cluster point is a connected nilpotent Lie group of homogeneous dimension at most $D$ equipped with a left-invariant sub-Finsler metric.
\end{theorem}

We emphasise that the proofs of Theorems \ref{thm:relative} and \ref{thm:conj2} (and, indeed, Theorem \ref{thm:bft}) build on the fundamental structure theorem for approximate groups proved by Breuillard, Green and Tao in \cite{bgt}. 
Our new results also crucially rely on more specific results about approximate groups that we established in our previous paper \cite{proper.progs}.

The dimension and the homogeneous dimension coincide for a connected abelian Lie group. Accepting for a moment that if the groups $\Gamma_n$ appearing in Theorem \ref{thm:relative} are abelian then every cluster point naturally comes with the structure of an abelian group, it follows that in the abelian setting we do have Theorem \ref{thm:conj2} under the weaker hypothesis of Theorem \ref{thm:relative}, as follows.

\begin{theorem}\label{thm:conj2Ab}
Let $(\Gamma_n,S_n)$ be a sequence of abelian Cayley graphs such that $|S_n^n|\ll n^D|S_n|$.  Then for every sequence $m_n\gg n$ the sequence $(\Gamma_{n},\frac{d_{S_{n}}}{m_n})$ is GH-relatively compact, and every cluster point is a connected abelian Lie group of dimension at most $D$ equipped with an invariant Finsler metric.
\end{theorem}
In fact, as well as a stronger statement, the abelian case admits a much shorter proof than the general case that nonetheless captures many of the ideas we use in the general case and so may serve as a useful introduction to the general case for first-time readers. We therefore give this short proof in \S\ref{sec:abelian}.

%\begin{remark*}The example in the Heisenberg group discussed just after  \cite[Theorem 1.9]{proper.progs}, which also features in \cite[Example 1.11]{tao}, shows that the hypothesis of Theorem \ref{thm:conj2} cannot be similarly weakened to  $|S_n^n|\ll n^D|S_n|$.
%\end{remark*}

\bigskip
\noindent\textsc{Application to scaling limits of vertex-transitive graphs of large diameter.} As a corollary of Theorem \ref{thm:conj2Ab} we obtain the following refinement of \cite[Theorems 1 \& 2]{bft}.
In \cite[Theorem 1]{bft}, Benjamini, Finucane and the first author show that if $(X_n)$ is a sequence of vertex-transitive graphs satisfying the large-diameter condition $|X_n|\ll\diam(X_n)^D$ then $(X_{n},\frac{d_{S_{n}}}{\diam(X_n)})$ is GH-relatively compact and every cluster point is a torus equipped with an invariant Finsler metric. In \cite[Theorem 2]{bft} they give a sharp bound on the dimension of the limiting torus under the assumption that the degree of the sequence $X_n$ is uniformly bounded. 
We are now able to lift this condition on the degree. Indeed, combining Theorem \ref{thm:conj2Ab} with \cite[Propositions 3.1.1 \& 3.3.1]{bft} as in the proof of \cite[Theorem 2]{bft}, we obtain the following result.
\begin{corollary}\label{cor:toruscase}
Let $(X_n)$ be a sequence of finite vertex-transitive graphs satisfying the large-diameter condition
\begin{equation}\label{eq:diambound}
|X_n|\ll\deg(X_n)\diam(X_n)^D,
\end{equation}
where $\deg(X_n)$ is the degree of $X_n$ (potentially going to infinity). 
Then the sequence $(X_{n},\frac{d_{S_{n}}}{\diam(X_n)})$ is GH-relatively compact, and every cluster point is a torus of dimension at most $D$ equipped with an invariant Finsler metric.
\end{corollary}

\noindent\textsc{The Hausdorff dimension of the scaling limit.} It is interesting to note that in both Theorems \ref{thm:relative} and \ref{thm:conj2} we use $D$ to control some {\it algebraic} invariant of the limit. Indeed, both the dimension and the homogeneous dimension depend only on the limiting Lie group, not on the limiting metric. It is therefore tempting to try to refine Theorems \ref{thm:relative} by relating $D$ to the Hausdorff dimension of the limit, which always lies between the dimension and the homogeneous dimension (see e.g.\ the proof of \cite[Theorem 13]{B'}). The following statement shows that this is in fact impossible in quite a strong sense.

\begin{prop}\label{prop:HausdorffNotCV} For every positive non decreasing function $f$ such that $\lim_{t\to \infty} f(t)=\infty$, there exists a sequence of symmetric generating subsets $S_n$ of the discrete Heisenberg group $\HHH(\Z)$ such that
 $|S_n|^n \ll f(n) n^3 |S_n|$, whilst 
$(\HHH(\Z),d_{S_n}/n)$ converges to the real Heisenberg group equipped with the left-invariant Carnot--Carath\'eodory metric associated to the $\ell^{\infty}$-norm on $\R^2$, whose Hausdorff dimension coincides with the homogeneous dimension (which is $4$).
\end{prop}
We refer to \S \ref{sec:Hausdorff} for the necessary background on Carnot--Carath\'eodory metrics.

It turns out that this proposition itself is sharp, in the sense that the condition $|S_n^n| \ll f(n) n^3 |S_n|$ cannot be improved to $|S_n|^n \ll n^3 |S_n|$.
This follows from the following general fact.

\begin{prop}\label{prop:LimitFinsler}
Suppose that the assumptions of Theorem \ref{thm:relative} are satisfied, and that the sequence $(\Gamma_{n},\frac{d_{S_{n}}}{m_n})$ converges to a connected Lie group $G$ of dimension exactly $D$ equipped with some left-invariant sub-Finsler metric $d$. Then $d$ is Finsler and $(G,d)$ has Hausdorff dimension $D$. 
\end{prop}

\bigskip
\noindent\textsc{Further growth of locally polynomially growing sets.} A key ingredient in the proof of Theorem \ref{thm:bft} was a result of Breuillard and the second author \cite[Theorem 1.1]{bt} implying, for example, that there exists a constant $C_D$ such that if \eqref{hom.poly.growth} holds for sufficiently large $n_1$ (depending on $D$) then $|S^{km}|\le C_D^k|S^m|$ for every $k\in\N$ and $m\ge n_1$. Tao subsequently gave a more precise description of the further growth of $S^m$, as follows.
\begin{theorem}[Tao {\cite[Theorem 1.9]{tao}}]\label{thm:tao}
Given $D>0$ there exists $N=N_{M,D}$ such that if $n\ge N$ and $S$ is a finite symmetric generating set for a group $G$ such that
\[
|S^n|\le Mn^D|S|
\]
then there exists a non-decreasing continuous piecewise-linear function $f:[0,\infty)\to[0,\infty)$ with $f(0)=0$ and at most $O_D(1)$ distinct linear pieces, each with a slope that is a natural number at most $O_D(1)$, such that
\[
\log|S^{mn}|=\log|S^n|+f(\log m)+O_D(1)
\]
for every $m\in\N$.
\end{theorem}

Tao phrases his argument in the language of non-standard analysis, both to streamline the presentation and to give easier access to some results of \cite{bgt} that are also phrased in that language. It turns out that some of the techniques we have developed in the present paper and its predecessor \cite{proper.progs} allow us to translate Tao's argument to a finitary setting (we emphasise that it is still essentially Tao's argument). Since the details are brief, we take this opportunity to present them in the appendix.

We caution that Theorem \ref{thm:tao} is still ineffective, since the argument uses \cite[Theorem 1.6]{bgt} as a black box (see Remark \ref{rem:bgt.location} for details). Nonetheless, this is now the only source of ineffectiveness in the argument, and so Theorem \ref{thm:tao} is effective for any class of group in which there is an effective version of \cite[Theorem 1.6]{bgt} (this includes residually nilpotent groups \cite{resid} and certain linear groups \cite{sol.lin,gill-helf}, for example). In particular, an effective version of \cite[Theorem 1.6]{bgt} in general would yield an effective version of Theorem \ref{thm:tao} in general.

\bigskip
\noindent\textsc{Notation.} 
%As we remarked above, underpinning all of what we do here are results on \emph{approximate groups} from \cite{bgt,proper.progs}. Given $K\ge1$, a finite set $A$ of a group $G$ is said to be a \emph{$K$-approximate group} if it is symmetric and contains the identity and there exists $X\subset G$ with $|X|\le K$ such that $A^2\subset XA$.  In the present paper we assume familiarity with some of the terminology and notation from our previous paper \cite{proper.progs}. In particular \matt{note to self: list the terms, but maybe don't define them.} We also adapt some of that terminology for use in a Lie group $G$. In particular, given elements $u_1,\ldots,u_d\in G$ and positive reals $L_1,\ldots,L_d$ we define the \emph{real ordered progression} $P_\ord^\R(u;L)$ via \[P_\ord^\R(u;L)=\{u_1^{\ell_1}\cdots u_d^{\ell_d}:\ell_i\in[-L_i,L_i]\},
% \] calling $d$ the \emph{rank} of this progression. We say that $(u;L)$ is in $C$-upper-triangular form \emph{over $\R$} if for every $i<j\le d$ and every $\eps_i,\eps_j\in[-1,1]$ we have \[ [u_i^{\eps_i},u_j^{\eps_j}]\subset P_\ord^\R\left(u_{j+1},\ldots,u_d;\textstyle{\frac{CL_{j+1}}{L_iL_j},\ldots,\frac{CL_d}{L_iL_j}}\right). \] \matt{Write $(e_{i_1},\ldots,e_{i_k};L)$ to mean $(e_{i_1},\ldots,e_{i_k};L_{i_1},\ldots,L_{i_k})$; i.e.\ automatically take the appropriate terms from $L$ if only have some terms of $e$.}
We follow the standard convention that if $X,Y$ are real quantities  and $z_1,\ldots,z_k$ are variables or constants then the expressions $X\ll_{z_1,\ldots,z_k}Y$ and $Y\gg_{z_1,\ldots,z_k}X$ each mean that there exists a constant $C>0$ depending only on $z_1,\ldots,z_k$ such that $X$ is always at most $CY$. 
We write $X \asymp_{z_1,\ldots,z_k} Y$ to mean that $X\ll_{z_1,\ldots,z_k}Y\ll_{z_1,\ldots,z_k}X$.

Moreover, the notation $O_{z_1,\ldots,z_k}(Y)$ denotes a quantity whose absolute value is at most a certain constant (depending on $z_1,\ldots,z_k$) multiple of $Y$, whilst $\Omega_{z_1,\ldots,z_k}(X)$ denotes a quantity that is at least a certain positive constant (depending on $z_1,\ldots,z_k$) multiple of $X$. Thus, for example, the meaning of the notation $X\le O(Y)$ is identical to the meaning of the notation $X\ll Y$. 

Here, and throughout this paper, we use the standard notation $AB=\{ab:a\in A,b\in B\}$, $A^n=\{a_1\cdots a_n:a_i\in A\}$ and $A^{-n}=\{a_1^{-1}\cdots a_n^{-1}:a_i\in A\}$. We also adopt the additional convention that given two subsets $A,B$ of a group we write $A\approx_{z_1,\ldots,z_k}B$ to mean that $A\subset B^{O_{z_1,\ldots,z_k}(1)}$ and $B\subset A^{O_{z_1,\ldots,z_k}(1)}$. If $A,B$ are subsets of a Lie algebra, we adopt the same convention but with additive notation, writing $A\approx_{z_1,\ldots,z_k}B$ to mean that $A\subset O_{z_1,\ldots,z_k}(1)B$ and $B\subset O_{z_1,\ldots,z_k}(1)A$.

\begin{definition}\label{def:qi} Given $C\geq 1$ and $K\geq 0$ and metric spaces $X,Y$, a map $f:X\to Y$ is said to be a \emph{$(C,K)$-quasi-isometry} if
\[C^{-1}d(x,y)-K\leq d(f(x),f(y))\leq Cd(x,y)+K,\]
for every $x,y\in X$, and every $y\in Y$ lies at distance at most $K$ from $f(X)$.
\end{definition}

\bigskip
\noindent\textsc{Overview of the paper.} In Section \ref{sec:background} we recall some notation and background material on approximate groups and progressions from our previous paper. In Section \ref{sec:GH-limit} we define Gromov--Hausdorff limits, explaining how to reduce their study to the study of \emph{ultralimits}, and then in Section \ref{sec:abelian} we give a direct proof of Theorem \ref{thm:conj2Ab}.

We give an overview of the general argument for Theorems \ref{thm:relative} and \ref{thm:conj2} in Section \ref{sec:overview}, reducing their proofs to two results labelled Proposition \ref{prop:reducConnected} and Theorem \ref{thm:homdimOfLimit}. We prove Proposition \ref{prop:reducConnected} in Section \ref{section:reducConnected}, and Theorem \ref{thm:homdimOfLimit} in Sections \ref{sec:normed} and \ref{sec:marked}.

Finally, we prove Propositions \ref{prop:HausdorffNotCV} and \ref{prop:LimitFinsler} in Section \ref{sec:Hausdorff}, and Theorem \ref{thm:tao} in the appendix.

\bigskip
\noindent\textsc{Acknowledgements.} The authors are grateful to Itai Benjamini for stimulating discussions. %\romain{As far as I am concerned, I only had interactions with Itai about this.}\matt{Feel free to write what you want to.}

\section{Background on approximate groups}\label{sec:background}
The strenghtening of Gromov's theorem due to Breuillard, Green and Tao described in the introduction relies on the observation that if $S$ satisfies \eqref{hom.poly.growth} then there is some $n\ge n_1^{1/2}$ such that $S^n$ is a so-called \emph{approximate group}. A finite subset $A$ of a group $G$ is said to be a \emph{$K$-approximate subgroup of $G$}, or simply a \emph{$K$-approximate group}, if it is symmetric and contains the identity and there exists $X\subset G$ with $|X|\le K$ such that $A^2\subset XA$.

An important result in approximate group theory is a remarkable theorem of Breuillard, Green and Tao \cite[Corollary 2.11]{bgt} describing the algebraic structure of an arbitrary approximate group in terms of certain objects called \emph{progressions}, versions of which we now define. This result underpins the approach of the present paper.

Following \cite{nilp.frei}, we define the \emph{ordered progression} on generators $u_1,\ldots,u_d\in G$ with lengths $L_1,\ldots,L_d$ is to be
\[
P_\ord(u;L):=\{u_1^{\ell_1}\cdots u_d^{\ell_d}:|l_i|\le L_i\}.
\]
If $P$ is an ordered progression and $H$ is a finite subgroup normalised by $P$, then we say that $HP$ is an \emph{ordered coset progression}.

Following \cite{bgt}, we say that the tuple $(u;L)=(u_1,\ldots,u_d;L_1,\ldots,L_d)$ is in \emph{$C$-upper-triangular form} if, whenever $1\le i<j\le d$, for all four choices of signs $\pm$ we have
\begin{equation}\label{eq:C-upp-tri}
[u_i^{\pm1},u_j^{\pm1}]\in P_\ord\left(u_{j+1},\ldots,u_d;\textstyle{\frac{CL_{j+1}}{L_iL_j},\ldots,\frac{CL_d}{L_iL_j}}\right).
\end{equation}
We say that an ordered progression is in $C$-upper-triangular form if the corresponding tuple is. We say that a coset ordered coset progression $HP$ is in $C$-upper-triangular form if the corresponding tuple is in $C$-upper-triangular form modulo $H$.

Given $m>0$, an ordered progression $P$ on the tuple $(u;L)=(u_1,\ldots,u_d;L_1,\ldots,L_d)$ is said to be \emph{$m$-proper with respect to a homomorphism $\pi:\langle P\rangle\to N$} if the elements $\pi(u_1^{\ell_1}\cdots u_d^{\ell_d})$ are all distinct as the $\ell_i$ range over those integers with $|\ell_i|\le mL_i$. The progression $P$ is said to be $m$-proper with respect to a subgroup $H\lhd\langle HP\rangle$ if $P$ is $m$-proper with respect to the quotient homomorphism $\langle HP\rangle\to\langle HP\rangle/H$. In this case we also say that the coset ordered coset progression $HP$ is \emph{$m$-proper}. If a coset ordered coset progression $HP$ is $m$-proper for every $m<0$ then we say it is \emph{infinitely proper}.

Having made these definitions, we can now state the result of Breuillard, Green and Tao. This result is essentially \cite[Corollary 2.11]{bgt}, although we state a version of it from our earlier paper \cite[Theorem 1.4]{proper.progs} in order to be compatible with our notation.
\begin{theorem}[Breuillard--Green--Tao]
Let $A$ be a $K$-approximate group. Then there exist an $\Omega_K(1)$-proper ordered coset progression $HP\subset A^4$, of rank and step at most $O_K(1)$ and in $O_K(1)$-upper-triangular form, and a set $X\subset\langle A\rangle$ with $|X|\ll_K1$ such that $A\subset XHP$.
\end{theorem}
This result enters the present paper via Proposition \ref{prop:inverse}.

We adapt some of the above terminology for use in a Lie group $G$. In particular, given elements $u_1,\ldots,u_d\in G$ and positive reals $L_1,\ldots,L_d$ we define the \emph{real ordered progression} $P_\ord^\R(u;L)$ via
\[
P_\ord^\R(u;L)=\{u_1^{\ell_1}\cdots u_d^{\ell_d}:\ell_i\in[-L_i,L_i]\},
\]
calling $d$ the \emph{rank} of this progression. We say that $(u;L)$ is in $C$-upper-triangular form \emph{over $\R$} if for every $i<j\le d$ and for all four choices of signs $\pm$ we have
\[
[u_i^{\pm1},u_j^{\pm1}]\subset P_\ord^\R\left(u_{j+1},\ldots,u_d;\textstyle{\frac{CL_{j+1}}{L_iL_j},\ldots,\frac{CL_d}{L_iL_j}}\right).
\]

Finally, given group elements $u_1,\ldots,u_d$ and positive reals $L_1,\ldots,L_d$, and given integers $i_1<\ldots<i_k\in[1,d]$, we often abbreviate $(u_{i_1},\ldots,u_{i_k};L_{i_1},\ldots,L_{i_k})$ to simply $(u_{i_1},\ldots,u_{i_k};L)$.

\section{GH-limits and ultralimits}\label{sec:GH-limit}

In this section we describe two different ways of defining a limit of a sequence of metric spaces: a \emph{Gromov--Hausdorff limit} and an \emph{ultralimit}. Underpinning the entire approach of this paper is the standard fact, stated below as Lemma \ref{lemPrelim:ultra}, that in the setting we are concerned with these two types of limit coincide.

We start by describing Gromov--Hausdorff limits, or \emph{GH-limits} as we call them from now on.
 \begin{definition}\label{GH}
Given a sequence $X_n$ of compact metric spaces we will say that $X_n$ GH-converges to $X$ if the $X_n$ have bounded diameter and if there exist maps $\phi_n:X_n\to X$ such that for all $\eps$, then for  $n$ large enough, 
\begin{itemize}
\item every point of $X$ is at $\eps$-distance of a point of $\phi_n(X_n)$;
\item $(1-\eps)d(x,y)-\eps\leq d(\phi_n(x),\phi_n(y))\leq  (1+\eps)d(x,y)+\eps$ for all $x,y\in X_n$.
 \end{itemize}
\end{definition}
GH-convergence also extends naturally to (not necessarily compact) locally compact pointed metric spaces \cite[\S3]{Gromov}, as follows.
\begin{definition}\label{GH2}
Given a  sequence $(X_n,o_n)$ of locally compact pointed metric spaces, $(X_n,o_n)$ is said to converge to the locally compact pointed metric space $(X,o)$ if for every $R>0$, the sequence of balls $B(o_n,R)$ GH-converges to $B(o,R)$.
\end{definition}

Gromov gave the following useful criterion for GH-precompactness.

\begin{lemma}[Gromov's compactness criterion, {\cite[Theorem 5.41]{BH}}]
A sequence of metric spaces $(X_n)$ is relatively GH-compact if and only if the $X_n$ have bounded diameter, and are ``equi-relatively compact": for every $\eps>0$, there exists $N\in \N$ such that for all $n\in \N$, $X_n$ can be covered by at most $N$ balls of radius $\eps.$
\end{lemma}
This fact readily extends to pointed metric spaces as follows. 
\begin{corollary}\label{lemPrelim:compact}
A sequence of pointed metric spaces $(X_n,o_n)$ is relatively GH compact if and only if for every $\eps>0$, there exists $N\in \N$ such that for all $n\in \N$, $B(o_n,1/\eps)$ can be covered by at most $N$ balls of radius $\eps.$
\end{corollary}
The following standard observation is particularly useful in light of Corollary \ref{lemPrelim:compact} (recall that a metric space is called \emph{homogeneous} if its group of isometries acts transitively).
\begin{lemma}\label{prop:doublingImpliesMetricDoubling}
Let $(X,d)$ be a homogeneous discrete metric space, and assume that for all $r>0$ and $x\in X$ we have $|B(x,2r)|\leq C|B(x,r)|$. Then $B(x,2r)$ is covered by $O_C(1)$ balls of radius $r$. 
\end{lemma}
\begin{proof}
Let $(x_i)_{i\in I}$ be a maximal $r$-separated (i.e.\ elements are at pairwise distance at least $r$) subset of $B(x,2r)$. Since the balls $B(x_i,r/3)$ are disjoint and contained in $B(x,3r)$, we deduce that $I$ has cardinality $O_C(1)$. On the other hand we deduce from maximality that the balls $B(x_i,r)$ cover $B(x,2r)$. 
\end{proof}

\begin{corollary}\label{cor:doubloingImpliesGHrc}
Let $(X_n,o_n)$ be a sequence of pointed discrete metric spaces such that  for every $n$, the metric space $X_n$ is homogeneous and satisfies some uniform doubling property: there exists  a sequence $\eps_n\to 0$ and a constant $C$ such that for all $n\in \N$, $r\geq \eps_n$, $|B(o_n,2r)|\leq C|B(o_n,r)|$. Then the sequence of pointed metric spaces $(X_n,o_n)$ is GH-relatively compact.
\end{corollary}

In \cite[\S3]{Gromov}, Gromov discusses functoriality of these notions of convergence. In particular, he mentions a notion of equivariant GH-convergence. In \cite[\S 3]{FY}, Fukaya and Yamaguchi provide a precise notion of GH-convergence for sequences  of triplets $(X_n,G_n,o_n)$, where $(X_n,o_n)$ are locally compact pointed metric spaces, and $G_n$ is a group of isometries of $X_n$. 
 They prove the important fact that if $(X_n,o_n)$ GH-converges to $(X,o)$ and if $G_n$ is a group of isometries of $X_n$ acting transitively, then in some suitable sense, a subsequence of the triplet $(X_n,G_n,o_n)$ converges to $(X,G,o)$ where $G$ is a subgroup of isometries of $X$ acting transitively. We will recover this fact in a special case, through the convenient notion of ultralimits. 
  
Before we can define ultralimits, we need to define ultra\emph{filters}. An \emph{ultrafilter} is a map from $\omega:\PP(\N)\to \{0,1\}$, such that $\omega(\N)=1$, and which is ``additive" in the sense that $\omega(A\cup B)=\omega(A)+\omega(B)$ for all $A$ and $B$ disjoint subsets of $\N$ \cite{Com}. It is easy to check that this is equivalent to the definition used in \cite{bgt}, for example.

Given an ultrafilter $\omega$, we say that a statement holds for \emph{$\omega$-almost every $n\in\N$} if the set $A$ of $n$ for which the statement holds satisfies $\omega(A)=1$. We write $x_n=_\omega y_n$ to mean that $x_n=y_n$ for $\omega$-almost every $n$; we write $x_n\ll_\omega y_n$ to mean that there exists some real number $C$ such that $x_n\le Cy_n$ for $\omega$-almost every $n$; and we write $O_\omega(x_n)$ for a sequence that is $\ll_\omega x_n$, write $\Omega_\omega(x_n)$ for a sequence that is $\gg_\omega x_n$, and write $o_\omega(x_n)$ for a sequence that is not in $\Omega_\omega(x_n)$.

Let us highlight a slightly subtle distinction between the notation $x_n\ll_\omega y_n$ and the asymptotic notation $x_n\ll y_n$. The notation $x_n\ll y_n$ implies the existence of some \emph{universal} constant $C$ such that $x_n\le Cy_n$ for all $n$, whereas the notation $x_n\ll_\omega y_n$ merely requires that for the sequence in question there exists some constant $C$ such that $x_n\le Cy_n$ for $\omega$-almost every $n$. Thus, for example, if we let $m$ be an arbitrary integer, and then defined $f(n)=m$ for every $n$, we would have $f(n)\ll_\omega1$, but we could not write $f(n)\ll1$ because for any choice of $C$ we would violate this assertion by choosing $m=\lceil C+1\rceil$. 

Ultrafilters are used to ``force" convergence of bounded sequences of real numbers. Given such a sequence $a_n$, its limit is defined to be the unique real number $a$ such that for every $\eps>0$ we have $|a_n-a|<\eps$ for $\omega$-almost every $n$. In this case we denote $\lim_{\omega} a_n=a$, or write $a_n\to_\omega a$. A trivial observation that is nonetheless extremely useful in this paper is the following.
\begin{lemma}\label{lem:finite.omega.const}
Let $(a_n)$ be a sequence of real numbers taking only finitely many different values. Then $a_n=_\omega\lim_\omega a_n$.
\end{lemma}

An ultrafilter is called non-principal if it vanishes on finite subsets of $\N$. Non-principal ultrafilters are known to exist, although this requires the axiom of choice \cite{Com}. From now on we fix some non-principal ultrafilter $\omega$.

\begin{definition}\label{def:ultralimitmetricspace}
Given a sequence of pointed metric spaces $(X_n,o_n)$, its ultralimit $(X_{\infty},o_{\infty})$ with respect to $\omega$ is the quotient of 
$$X_{\omega}=\{(x_n)\in \Pi_n X_n:  d(x_n,o_n)\ll_\omega1\}$$
%\Matt{We should define the $=_{\omega}O(1)$ and $=_{\omega}o(1)$ notation.} 
by the equivalence relation $x_n\sim y_n$ if $d(x_n,y_n)=_{\omega} o(1)$. It is equipped with a distance defined by $d_{\infty}((x_n),(y_n))=\lim_{\omega}d(x_n,y_n).$
\end{definition}
It is a basic fact that a sequence $a_n\in \R$ converging to $a\in\R$ in the usual sense also converges to $a$ with respect to $\omega$, in the sense that $\lim_{\omega}a_n=a$. A similar fact holds for ultralimits of pointed metric spaces, as follows.

\begin{lemma}[{\cite[Exercise 5.52]{BH}}]\label{lemPrelim:ultra} If a sequence of pointed metric spaces converges in the pointed GH sense to $X$, then its ultralimit with respect to $\omega$ is isometric to $X$. 
\end{lemma}

Assume from now that the sequence $(X_n,o_n)$ is a sequence of groups $(G_n,e_n,d_n)$ where $e_n$ is the identity element and $d_n$ is some left-invariant metric. Then $G_{\omega}$ is a group that acts transitively on $G_{\infty}$. The stabiliser of the identity
is $$G_{\omega,0}=\{(x_n)\in \Pi_n G_n:d(x_n,e_n)=_{\omega} o(1)\}.$$ 
If the $G_n$ are abelian groups, then $G_{\infty}$ naturally comes with an abelian group structure as well. More generally, 
 $G_{\omega,0}$ is a normal subgroup of $G_{\omega}$ if and only if for all sequences $(g_n)$ and $(h_n)$ such that $d(g_n,e_n)=_{\omega}O(1)$ and  $d(h_n,e_n)=_{\omega}o(1)$ we have $d([g_n,h_n],e_n)=_{\omega}o(1)$. In this case $(G_{\infty},d_{\infty})$ naturally comes with a metric group structure.

We close this section by recording the following basic fact about maps between metric spaces that behave well under ultra-limits. The proof is straightforward from the definition of an ultralimit, so we leave it as an exercise. 
 \begin{prop}\label{prop:ultralimitPhi}
 Let $\omega$ be a non-principal ultrafilter on $\N$.
 Assume that $(X_n,o_n)$ and $(X'_n,o'_n)$ are sequences of pointed metric spaces, and let $\phi_n:X_n\to X'_n$ be a sequence of maps.  We denote by $(X_{\infty},o_{\infty})$ and $(X'_{\infty},o'_{\infty})$ the respective ultralimits of $(X_n,o_n)$ and $(X'_n,o'_n)$. 
 \begin{itemize}
  \item[(i)]
Assume  that there exists $C\geq 1$ such that
 for every sequence $u_n,v_n\in B(o_n,O_\omega(1))$ we have
 $$d(\phi (u_n),\phi(v_n))\leq_\omega Cd(u_n,v_n)+o_\omega(1).$$
 Then there exists a $C$-Lipschitz map  $\phi_{\infty}:X_{\infty}\to Y_{\infty}$, defined as $\phi_{\infty}(\lim_{\omega}u_n)=\lim_{\omega}\phi(u_n)$ for all $u_n\in B(o_n,O_\omega(1))$.

\item [(ii)] Assume that in addition to (i) that there exist $\alpha,\beta>0$ such that for every sequence $u'_n\in B(o'_n,\beta)$ there exist $u_n\in B(o_n,\alpha)$ such that $d(u'_n,\phi_n(u_n))=o_\omega(1)$. Then $\phi_{\infty}(B(o_{\infty},\alpha))$ contains $B(o'_{\infty},\beta)$.
  \end{itemize}
 \end{prop}
 
 \begin{corollary}\label{corultralimitPhi}
Let $\phi_n:(G_n,d_n)\to (H_n,\delta_n)$ be a sequence of morphisms between two groups equipped with left-invariant metrics such the assumptions of Proposition \ref{prop:ultralimitPhi} (ii) are satisfied. Assume, moreover, that $G_{\omega,0}$ and $H_{\omega,0}$ are normal subgroups of respectively $G_{\omega}$ and $H_{\omega}$, and hence that the ultralimits $G_{\infty}$ and $H_{\infty}$ naturally come with group structures (this is the case, for example, if $G_n$ and $H_n$ are abelian). Then $\phi_{\infty}$ is a Lipschitz open homomorphism, and if $H_{\infty}$ is connected then $\phi_{\infty}$ is surjective, and therefore is an isomorphism. 
 \end{corollary}
 
We end this section with the following standard observation.
\begin{remark}\label{rem:ScalingLength}
Recall that an ultralimit of a sequence of geodesic metric spaces is geodesic (see for instance \cite[\S 7.5]{Roe}). On the other hand if $(X_n)$
 is a sequence of connected graphs equipped with their usual geodesic distances $d_n$, $m_n$ is a sequence going to infinity, and $\omega$ is a non-principal ultrafilter, the $m_n$-rescaled ultralimit of $(X_n,d_n)$ is isometric to the $m_n$-rescaled ultralimit of the $0$-skeletons of $X_n$ equipped with the induced metric. In particular the latter is geodesic. Applied to our setting, if $(G_n,d_{S_n})$ is a sequence of groups equipped with word metrics, then $\lim_{\omega}(G_{n},d_{S_n}/m_n)$ is geodesic (hence connected).
 \end{remark}

\section{The abelian case}\label{sec:abelian}
In this section we provide the short proof of Theorem \ref{thm:conj2Ab}. This provides a gentle introduction to the methods of the rest of the paper, but is completely independent of the subsequent sections; in particular, the reader interested only in Theorem \ref{thm:conj2} could easily skip the present section.

Given an abelian group $G$, a \emph{coset progression of rank $d$} is a subset of $G$ of the form $PH$, where $H< G$ is a finite subgroup and $P=P(u,L)$, with $u=(u_1,\ldots, u_d)$ for some $u_i\in G$ and $L=(L_1,\ldots, L_d)$ for some $L_i\geq 1$, is defined via
\[
P=\left\{u_1^{\ell_1}\cdots u_d^{\ell_d} :|\ell_j|\leq L_j\right\}.
\]
Given $m>0$, this coset progression is called \emph{$m$-proper} if the morphism $\phi: \Z^d\to G/H$ mapping the canonical basis $x=(x_1,\ldots, x_i)$ to $u$ is injective in restriction to $P(x,mL)$.

The following lemma is both standard and trivial.
\begin{lemma}\label{lem:ab.prog.doub}
Let $P$ be a coset progression of rank at most $d$ inside an abelian group $G$. Then for every $k\in\N$ there exists $X_k\subset G$ with $|X_k|\ll_d1$ such that $P^{2k}\subset X_kP^k$.
\end{lemma}

We recall the following result from our previous paper.
\begin{theorem}[{\cite[Theorem 1.9]{proper.progs}}]\label{thm:ab.Frei.proper}
Let $A$ be a $K$-approximate group such that $\langle A\rangle$ is abelian. Then for every $m>0$ there exists an $m$-proper coset progression $HP$ of rank at most $K^{O(1)}$ such that 
\[ A\subset HP\subset A^{O_{K,m}(1)}.\]
\end{theorem}

%\subsection{Proof of Theorem \ref{thm:conj2Ab}}
%\begin{theorem}\label{thm:conj2Ab}Let $(\Gamma_n,S_n)$ be a sequence of abelian Cayley graphs such that $|S_n^n|=O(n^D|S_n|)$.  Then for every sequence $m_n\gg n$ the sequence $(\Gamma_{n},\frac{d_{S_{n}}}{m_n})$ is GH-relatively compact, and every cluster point is a connected abelian Lie group of dimension at most $D$ equipped with an invariant Finsler metric.\end{theorem}

\begin{lemma}\label{lem:ab.phole}
Under the assumptions of Theorem \ref{thm:conj2Ab}, for all but finitely many $n\in\N$ there exists a sequence $k_n\in\N$ with $n^{1/4}\leq k_n\leq n^{1/2}$ and a $2$-proper coset progression $H_nP_n$ of rank $d_n\ll_D1$ such that
\begin{equation}\label{eq:SnHnPn}
S_n^{k_n}\subset H_nP_n\subset S_n^{O_D(k_n)},
\end{equation}
and such that, writing $P_n=P_{ord}(u^{(n)},L^{(n)})$, the rank $d_n$ is minimal in the sense that for each $1\leq i\leq d_n$, if $\bar{L}^{(n)}=(L_1^{(n)},\ldots,L_{i-1}^{(n)},0,L_{i+1}^{(n)},\ldots,L_{d_n}^{(n)})$ then $S_n^{k_n}$ is not contained in $H_nP(u^{(n)},\bar{L}^{(n)})$.
\end{lemma}

\begin{proof}
Applying \cite[Lemma 8.4]{proper.progs} with $\lfloor n^{1/2}\rfloor$ in place of $n$ gives, for all but finitely many $n$, an integer $\ell_n$ satisfying $n^{1/4}\leq\ell_n\leq\frac{1}{2}n^{1/2}$ such that $|S_n^{5\ell_n}|/|S_n^{\ell_n}|\ll_D1$.  It then follows from \cite[Lemma 2.2]{bt} that $S_n^{2\ell_n}$ is an $O_D(1)$-approximate group. Setting $k_n=2\ell_n$, the existence of a $2$-proper coset progression $H_nP_n$ satisfying \eqref{eq:SnHnPn} then follows from Theorem \ref{thm:ab.Frei.proper}. We obtain the minimality of $d_n$ simply by deleting dimensions whenever this is possible without violating \eqref{eq:SnHnPn}.
\end{proof}

We will eventually apply the following lemma to the coset progressions coming from Lemma \ref{lem:ab.phole}.

\begin{lemma}\label{lem:largeLigGrowth}
Let $\Gamma$ be an abelian group with finite symmetric generating set $S$ containing the identity, let $k\in\N$, and let $HP=HP(u;L)$ be a $2$-proper coset progression of rank $d$ such that $S^k\subset HP$ but such that
\begin{equation}\label{eq:L.min}
S^k\not\subset HP(u;L_1,\ldots,L_{i-1},0,L_{i+1},\ldots,L_d)
\end{equation}
for every $i$. Then $L_i\ge k$ for every $i$ and $S\subset HP(u;L/k)$.
\end{lemma}

\begin{proof}
Let $x_1,\ldots x_d$ be a basis of $\Z^d$, write $Q=P(x,L)$, and let $\varphi:\Z^d\to \Gamma$ be the homomorphism mapping $x_i$ to $u_i$ for each $i$. The properness of $HP$ implies that there exists a unique subset $U\subset Q$ such that $\varphi(U)=S$ modulo $H$. 

We first claim that $U^k\subset Q$. Indeed, let $\ell\in\N$ be minimal such that there exists $z_1,\ldots, z_\ell\in U$ with $z=z_1\cdots z_\ell\notin Q$, noting that by minimality we have $z_1\cdots z_{\ell-1}\in Q$ and hence $z\in Q^2\setminus Q$. If $\ell\le k$ we would have $\varphi(z)\in S^k\subset HP$, and so there would exist $z'\in Q$ with $\varphi(z)=\varphi(z')$ mod $H$, contradicting the $2$-properness of $P$. Thus $\ell>k$, and in particular
\begin{equation}\label{Uk.in.Q}
U^k\subset Q,
\end{equation}
as claimed.

This immediately implies that $S\subset HP(u;L/k)$, as required. Combined with \eqref{eq:L.min} and the properness of $HP$, it also implies that for every $i$ satisfying $1\leq i\leq d$ there exists $z\in U$ such that the $i$th coordinate of $z$ is non-zero. It then follows from \eqref{Uk.in.Q} that $z^k$, whose $i$th coordinate is at least $k$, belongs to $Q$, and so $L_i\ge k$, as required.
\end{proof}

\begin{lemma}\label{lem:progressionPrecompact}
Let $d\in\N$, and let $(\Gamma_n,V_n)$ be a sequence of abelian Cayley graphs such that for every $n$ the generating set $V_n:=P(u^{(n)},K^{(n)})H_n$ is a coset progression of rank at most $d$. Then for every sequence $m_n\to\infty$ the sequence of (pointed) metric spaces $(\Gamma_n,d_{V_n}/m_n)$ is GH-relatively compact. Moreover, for every non-principal ultrafilter $\omega$ the ultralimit $\lim_{\omega}(\Gamma_n,d_{V_n}/m_n)$ is a connected abelian Lie group equipped with a Finsler metric. 
\end{lemma}
\begin{proof}
It follows from Corollary \ref{lemPrelim:compact} and Lemma \ref{lem:ab.prog.doub} that $(\Gamma_n,d_{V_n}/m_n)$ is relatively compact. 
Observe that since $H_n$ is contained in a ball of radius $1=o(m_n)$, the limit is not affected by modding out by $H_n$, so we can simply assume that it is trivial. Note also that by Lemma \ref{lem:finite.omega.const} we may assume that the rank of $V_n$ is equal to $d$ for every $n$. Consider then the morphism $\pi_n: \Z^d\to \Gamma_n$ mapping the canonical basis $x=(x_1,\ldots,x_d)$ to $u^{(n)}$, and write $W_n:=P(x^{(n)},K^{(n)})$. Being a length space, $\lim_{\omega}(\Gamma_n,d_{V_n}/m_n)$ is in particular connected, therefore Corollary \ref{corultralimitPhi} and Remark \ref{rem:ScalingLength} implies that $\lim_{\omega}\phi_n:\lim_{\omega}(\Z^d,d_{W_n}/m_n)\to \lim_{\omega}(\Gamma_n,d_{V_n}/m_n)$ exists and is a surjective morphism.  It is therefore sufficient to prove that $\lim_{\omega}(\Z^d,d_{W_n}/m_n)$ is a  normed real vector space of dimension $d$. Note that $W_n^{\R}=P^{\R}(x^{(n)},K^{(n)})$ is the convex hull of $W_n$ in $\R^d$, and by \cite[Lemma 2.2.5.]{bft}, for all $k\in \N$
\[(W_n^{\R})^k\subset W_n^k(W_n^{\R})^{O_d(1)}.\]
Hence, $(\Z^d,W_n)\hookrightarrow (\R^d,W_n^{\R})$ is a $(O_d(1),O_d(1))$-quasi-isometry. 
We deduce from Corollary \ref{corultralimitPhi} that $\lim_{\omega}(\Z^d,d_{W_n}/m_n)$ is  bi-Lipschitz isomorphic to $\lim_{\omega}(\R^d,d_{W^{\R}_n}/m_n)$. Observe that $(\R^d,W_n^{\R})$ is isometric to $(\R^d,W^{\R})$, where $W^{\R}=P^{\R}(x^{(n)},1))$, and one easily checks that $\lim_{\omega}(\R^d,d_{W^{\R}/m_n})$ is isometric to $(\R^d,\|\cdot\|)$, where the norm $\|\cdot\|$ has $W^{\R}$ as unit ball. In conclusion, $\lim_{\omega}(\Z^d,d_{W_n}/m_n)$ is bi-Lipschitz isomorphic to $\R^d$, and since its metric is invariant and geodesic (by Remark \ref{rem:ScalingLength}), it actually comes from a norm (see the proof of \cite[Theorem 2.2.4]{bft}).
\end{proof}

\begin{lemma}\label{lem:radiusFreedom}
Let $d\in\N$, let $\omega$ be a non-principal ultrafilter, and let $(\Gamma_n,V_n)$ be a sequence of abelian Cayley graphs such that for $\omega$-almost every $n$ the generating set $V_n:=P(u^{(n)},K^{(n)})H_n$ is a coset progression of rank $d$. Then, on permuting the indices of $u_i^{(n)}$ for every such $n$ if necessary, there exists $d'\leq d$ such that the coset progression $V'_n:=P(u,K^{(n)}_1,\ldots,K^{(n)}_{d'},0,\ldots,0))H_n$ is $\Omega_{\omega}(m_n)$-proper and, writing $\Gamma'_n=\langle V'_n\rangle$ the inclusion $(\Gamma_n',V_n')\hookrightarrow (\Gamma_n,V_n)$ is a $1$-Lipschitz $(1,o_\omega(m_n))$-quasi-isometry.   
\end{lemma}
\begin{proof} 
The lemma is trivial if $d=0$, so we may assume that $d\ge1$ and, by induction, that the lemma holds with $d-1$ in place of $d$.

If $V_n$ is $\Omega_{\omega}(m_n)$-proper the lemma is trivially satisfied, so we may assume that there exists a sequence $r_n=o_{\omega}(m_n)$ and, for $\omega$-almost every $n$, an element $z^{(n)}\in P(x^{(n)},K^{(n)})^{r_n}$ such that, writing $x_1,\ldots x_d$ for the standard basis of $\Z^d$ and $\varphi_n:\Z^d\to\Gamma_n$ for the homomorphism mapping $x_i$ to $u_i^{(n)}$ for each $i$, we have $\varphi_n(z^{(n)})=1$. On permuting the coordinates of $z^{(n)}$ we may assume that $|z_d^{(n)}|/K^{(n)}_d=\max_i|z_i^{(n)}|/K^{(n)}_i$.

Let $\Lambda_n$ be the subgroup of $\Gamma_n$ generated by $W_n=P(u,K^{(n)}_1,\ldots,K^{(n)}_{d-1},0))$. We claim that, for every $r\ge0$, every element $v$ of length $r$ in $(\Gamma_n,V_n)$ lies at distance at most $r_n+1$ (with respect to $d_{V_n}$) from an element $y\in\Lambda_n$ of length $r+r_n$ with respect to $d_{W_n}$. By definition, such a $v$ is of the form $v=u_1^{j_1}\ldots u_d^{j_d}h$ with $\max_i |j_i|/K_i^{(n)}=r$, and $h\in H_n$. Let $q$ and $t$ be integers such that $j_d=qz^{(n)}_d+t$, with $|t|<|z^{(n)}_d|$ and $0\leq q\leq |j_d|/|z^{(n)}_d|$. Now, using that $\varphi_n(z^{(n)})=1$, we obtain
\[v=yu_d^{t}h,\]
where $y=u_1^{j_1-qz^{(n)}_1}\ldots u_{d-1}^{j_{d-1}-qz^{(n)}_{d-1}}$. In particular, $v$ lies at distance at most $|t|/K^{(n)}_d+1\leq |z^{(n)}_d|/K^{(n)}_d+1\leq r_n+1$ from $y$. On the other hand, $y$ lies in the ball of radius
\begin{align*}
\max_i |j_i|/K^{(n)}_i+q\max_i|z^{(n)}_{i}|/K^{(n)}_i
      &= r+q|z^{(n)}_{d}|/K^{(n)}_d\\
     &\leq r+r_n,
\end{align*}
which proves the claim. This implies that the inclusion $(\Lambda_n,W_n)\hookrightarrow(\Gamma_n,V_n)$ is a $1$-Lipschitz  $(1,o_\omega(m_n))$-quasi-isometry. We conclude by applying the induction hypothesis to $(\Lambda_n,W_n)$. 
\end{proof}

\begin{proof}[Proof of Theorem \ref{thm:conj2Ab}]
Let $k_n$, $H_n$, $P_n$ and $d_n$ be as given by Lemma \ref{lem:ab.phole} for all but finitely many $n$, and note that by Lemma \ref{lem:largeLigGrowth} we may assume that for every such $n$ we have $L_i^{(n)}\ge_\omega k_n$ for every $i$ and $S_n\subset H_nP(u^{(n)};L^{(n)}/k_n)$. Combined with \eqref{eq:SnHnPn}, this implies that setting $K_i^{(n)}=L_i^{(n)}/k_n\geq 1$ for every $i$ and writing $V_n=P(u^{(n)},K^{(n)})H_n$ we have
\begin{equation}\label{eq:SV}
V_n^{k_n}\approx_D S_n^{k_n}.
\end{equation}
This implies in particular that the word metrics on $\Gamma_n$ associated respectively to $S_n^{m_n/k_n}$ and to $V_n^{m_n/k_n}$ are $O_D(1)$-bi-Lipschitz equivalent. Lemma \ref{lem:progressionPrecompact} therefore implies that $(\Gamma,d_{S_n}/m_n)$ is GH-precompact, as required.

Lemma \ref{lemPrelim:ultra} implies that given a cluster point $(X,d_X)$ of $(\Gamma,d_{S_n}/m_n)$ we may fix a non-principal ultrafilter $\omega$ and assume that $(\Gamma,d_{S_n}/m_n)\to_\omega(X,d_X)$. We claim that we may also assume that $S_n=_\omega V_n$. We first consider the ultralimits. On the one hand, observe that by  Corollary \ref{corultralimitPhi}
the identity map $(\Gamma,d_{S_n}/m_n)\to (\Gamma,d_{S_n^{k_n}}/(m_n/k_n))$ (resp.\ $\Gamma,d_{V_n}/m_n)\to (\Gamma,d_{V_n^{k_n}}/(m_n/k_n))$) induces an isometric isomorphism between the ultralimits. Moreover, since the identity $(\Gamma,d_{S_n^{k_n}}/(m_n/k_n))\to  (\Gamma,d_{V_n^{k_n}}/(m_n/k_n))$ is $O_D(1)$-bi-Lipschitz by (\ref{eq:SV}), by Corollary \ref{corultralimitPhi} it induces a bi-Lipschitz isomorphism between the ultralimits, which are both geodesic by Remark \ref{rem:ScalingLength}. 
Composing these maps, we deduce that the identity map $(\Gamma,d_{S_n}/m_n)\to (\Gamma,d_{V_n}/m_n)$ induces a bi-Lipschitz isomorphism between the ultralimits. Since a geodesic invariant metric on an abelian connected Lie group is Finsler \cite{B} (or see the proof of \cite[Theorem 2.2.4.]{bft}), we deduce that if one ultralimit has a Finsler metric then so does the other.

Second, observe that as $S_n\subset V_n$, we have $|S_n|\leq |V_n|$. On the other hand, since $k_n=o(n)$, (\ref{eq:SV}) implies that there exists $c\gg_D1$ such that $V_n^{cn}\subset S_n^n$, which combines with Lemma \ref{lem:ab.prog.doub} to imply that $|V_n^n|\ll |S_n^n|$. The assumption that $|S_n^n|\ll n^D|S_n|$ therefore implies that $|V_n^n|\ll n^D|V_n|$, and so we may assume that $S_n=V_n$, as claimed.

Since $d_n\ll_D1$, Lemma \ref{lem:finite.omega.const} implies that there exists $d\ll_D1$ such that $d_n=_\omega d$. Lemma \ref {lem:radiusFreedom} then reduces matters to the case where $S_n$ is $\Omega_{\omega}(m_n)$-proper. One consequence is that $S_n^n\gg_{\omega} n^{d}|S_n|$, so that $d\le D$. A second consequence is that, writing $W_n=P(x,K^{(n)})$, the morphism $\lim_{\omega} (\Z^d,d_{W_n}/m_n)\to \lim_{\omega}(\Gamma_n,d_{S_n}/m_n)$, which is surjective by Corollary \ref{corultralimitPhi} is injective in restriction to a ball of radius $\Omega_\omega(1)$. By the end of the proof of Lemma \ref{lem:progressionPrecompact}, we have that $\lim_{\omega} (\Z^d,d_{W_n}/m_n)$ is a normed vector space of dimension $d$. Hence the dimension 
of any cluster point of $(\Gamma_n,d_{S_n}/m_n)$ equals $d\le D$, as required.
\end{proof}

\section{Progressions in free nilpotent groups}\label{sec:free.nilp}

Throughout this paper $N=N_{d,s}$ denotes the free $s$-nilpotent group on $d$ generators $x=(x_1,\ldots x_d)$, and $L=(L_1,\ldots, L_d)$ is a tuple of positive integers. 
Let $N^{\R}=N_{d,s}^{\R}$ be the Malcev completion of $N$, namely the free $s$-nilpotent Lie group of rank $d$, and let $\n$ be its Lie algebra.

We recall that the Baker--Campbell--Hausdorff formula states that for elements $X,Y$ in a Lie algebra we have
\begin{equation}\label{eq:bch}
\textstyle\exp(X)\exp(Y)=\exp(X+Y+\frac{1}{2}[X,Y]+\frac{1}{12}[X,[X,Y]]+\cdots).
\end{equation}
The precise values of the rationals appearing later in the series \eqref{eq:bch} are not important for our arguments; all that matters is that in a nilpotent Lie group the series is finite and depends only on the nilpotency class of the group.

For each $i$ we set $f_i=\log x_i$. Following \cite[\S11.1]{hall} and \cite[\S1]{bg}, we extend $f_1,\ldots,f_d$ to a list $\overline f=f_1,\ldots,f_r$ of so-called \emph{basic commutators} in the $f_i$. We first define general commutators in the $f_i$ recursively, starting by declaring $f_1,\ldots,f_d$ themselves to be commutators. As part of the recursive definition, we assign to each commutator $\alpha$ a \emph{weight vector} $\chi(\alpha)=(\chi_1(\alpha),\ldots,\chi_r(\alpha))$; for $f_1,\ldots,f_d$, these weight vectors are simply given by $\chi_i(f_i)=1$ and $\chi_j(f_i)=0$ for $i\ne j$. Now, given commutators $\alpha$ and $\beta$ whose weight vectors have already been defined, the Lie bracket $[\alpha,\beta]$ is also defined to be a commutator, with weight vector $\chi(\alpha)+\chi(\beta)$. We define the \emph{total weight} $|\chi(\alpha)|$ of a commutator $\alpha$ to be $\|\chi(\alpha)\|_1$.

We also now declare $f_1,\ldots,f_d$ to be \emph{basic} commutators. Having defined the basic commutators $f_1,\ldots,f_m$ of total weight less than $k$, we define the basic commutators of total weight $k$ to be those commutators of the form $[f_i,f_j]$ with $i>j$ and $|\chi(f_i)|+|\chi(f_j)|=k$, and such that if $f_i$ is of the form $[f_s,f_t]$ then $j\ge t$. We order these arbitrarily subject to the constraint that basic commutators with the same weight vector are consecutive, and abbreviate $\chi(i)=\chi(f_i)$. Note that the arbitrariness of the order implies that the list of basic commutators is not uniquely defined. We caution that the basic commutators in the $f_i$ are not in general equal to the logarithms of the basic commutators in the $x_i$. It is well known that $(f_1,\ldots,f_r)$ is a basis of $\n$, and more generally that the basic commutators of total weight at least $i$ form a basis for $\n_i$, the $i$th term of the lower central series of $\n$; see \cite[\S11]{hall} or \cite{Hall2}.

As in our previous paper \cite{proper.progs}, given $L\in\R^d$ and $\chi\in\N^d$ we use the notation $L^\chi$ to denote the tuple $L_1^{\chi_1}\cdots L_d^{\chi_d}$. Thus $B_\R(\overline f;L^\chi)$ is the set of linear combinations of $\sum_i\lambda_i f_i$ where  $|\lambda_i|\leq L^{\chi(i)}$. We also write $Lf=(L^{\chi(1)}f_1,\ldots,L^{\chi(r)}f_r)$. We remark that $B_\R(\overline f;L^\chi)$ is the continuous version of the \emph{nilbox} $\mathfrak{B}(f_1,\ldots,f_d;L)$ appearing in \cite{bg}.

\begin{lemma}\label{lem:real.B.upper-tri}
For every $L\in\N^d$ we have $[B_\R(\overline f;L^\chi),B_\R(\overline f;L^\chi)]\subset O_{d,s}(1)B_\R(\overline f;L^\chi)$.
\end{lemma}
\begin{proof}
Since $B_\R(\overline f;L^\chi)=B_\R(\overline{Lf};1)$, we may assume that $L=1$. The lemma then follows from the fact that $B_\R(\overline f;1)$ is a compact neighbourhood of the origin.
%Define as in \cite[\S3]{proper.progs} a partial order on the possible weight vectors of commutators by writing $\chi\ge\chi'$ if $\chi_i\ge\chi'_i$ for every $i$. Given a commutator $\alpha=\alpha(f_{i_1},\ldots,f_{i_t})$, define $\iota(\alpha)\in\N^r$ by setting $\iota_j(\alpha)=1$ if $\chi(j)\ge\alpha$, and $\iota_j(\alpha)=0$ otherwise. It follows from \cite[Lemma 3.6]{proper.progs} that
%\[
%\alpha\in O_{d,s}(1)B_\R(f;\iota(\alpha)),
%\]
%and hence that for $\lambda_1,\ldots,\lambda_t\in\R$ that
%\[
%\alpha(\lambda_1f_{i_1},\ldots,\lambda_tf_{i_t})\subset O_{d,s}(\lambda_1\cdots\lambda_t)B_\R(f;\iota(\alpha)),
%\]
%which immediately implies the desired result.
%\matt{If this is now okay you can delete this comment.}
\end{proof}

\begin{lemma}\label{lem:freeequivalences}
For every $n\in\N$ we have
\begin{itemize}
\item[(i)] $P_\ord^{\R}(x,L)^n\approx_{d,s}\exp(B_\R(\overline f;(nL)^\chi)\approx_{d,s}\exp(B_\R(f;nL))$,
\item[(ii)] $P_\ord^{\R}(x,L)^n\approx_{d,s}P_\ord^{\R}(x,nL)$, and
\item[(iii)] $\exp(B_\R(\overline f;(nL)^\chi)\subset\exp(B_\R(\overline f;L^\chi)^{O_{d,s,n}(1)}$.
\end{itemize}
\end{lemma}

\begin{proof}
Observe that for all $n\in \N$ we have $P_\ord^{\R}(x,nL)=P_\ord^{\R}(x^L,n)$, $B_\R(\overline f;(nL)^\chi)=B_\R(\overline{Lf};n^\chi)$, and $B_\R(f;nL)=B_\R(Lf;n)$, which reduces the the proof of these statements to the case where $L=1$. Then (iii) and the $n=1$ case of (i) simply follow from the fact that all terms are compact generating sets (explicit constants can also be obtained through the Baker--Campbell--Hausdorff formula); the general case of (i) follows from the proof of \cite[Theorem II.1]{Gui}. Finally, (i) implies that $P_\ord^{\R}(x,nL)\approx_{d,s}\exp(B_\R(\overline f;(nL)^\chi)\approx_{d,s}P_\ord^{\R}(x,L)^n$, which implies (ii).
%\matt{I added (iii) for use in Section 4 --- does it look right to you? If so you can just delete this comment rather than replying to it.}\romain{I agree that this should be true, but do we really need it?} \matt{I used it in the proof of Lemma \ref{lem:convexReal}.}
\end{proof}

\begin{prop}\label{prop:FreeReducCase}
%Let $N=N_{d,s}$ be the free $s$-nilpotent group on $d$ generators $x=(x_1,\ldots x_d)$, and $L=(L_1,\ldots, L_d)$ be a sequence of positive integers.  Then f
For all $k\in \N$ we have
\begin{equation}\label{eq:FreeReducCase1}
P_\ord^{\R}(x;L)^k\subset  P_\ord(x^L;1)^{O_{d,s}(k)}P_\ord^{\R}(x;L)^{O_{d,s}(1)}\subset P_\ord(x;L)^{O_{d,s}(k)}P_\ord^{\R}(x;L)^{O_{d,s}(1)}.
\end{equation}
%\Matt{Should the two instances of $P_\ord^{\R}(x;L)$ be $P_\ord^{\R}(x;L)^{O(1)}$, or even $P_\ord^{\R}(x;L)^{O_{d,s}(1)}$?}
Moreover,
\begin{equation}\label{eq:FreeReducCase2}
P_\ord^{\R}(x;L)^k\cap N\subset  P_\ord(x;L)^{O_{d,s}(k)}.
\end{equation}
\end{prop}

\begin{proof}
The second inclusion of \eqref{eq:FreeReducCase1} is trivial. To prove the first inclusion, note first that applying the automorphism of $N_{\R}$ that maps $x_i^{L_i}$ to $x_i$ reduces the statement to the case where $L_i=1$, and so it suffices to prove that
$$P_\ord^{\R}(x;1)^k\subset  P_\ord(x;1)^{O_{d,s}(k)}P_\ord^{\R}(x;1)^{O_{d,s}(1)}.$$
%\Matt{Should the $P_\ord^{\R}(x;1)^{O(1)}$ be $P_\ord^{\R}(x;1)^{O_{d,s}(1)}$?}

Let $D$ be a compact subset such that $N D=N^{\R}$. 
Observe that the map $\phi:N_{\R} \to N$ that sends $g\in N_{\R}$ to the unique $\gamma\in N$ such that $g\in \gamma D$ is a left-inverse of the inclusion $N\to N_{\R}$. Since the latter is a quasi-isometry for the word metric associated to $P_\ord^{\R}(x;1)$ and $P_\ord(x;1)$, we deduce that  $\phi$ itself is a quasi-isometry. In particular, the ball $P_\ord^{\R}(x;1)^k$ of radius $k$ must me contained in the preimage  by $\phi$ of a ball of radius $O_{d,s}(k)$.
%\Matt{Where did this $\overline{P}$ come from? i.e. the nilcomplete version?}
This implies that 
$$P_\ord^{\R}(x;1)^k\subset  P_\ord(x;1)^{O_{d,s}(k)}D,$$
and so the first inclusion of \eqref{eq:FreeReducCase1} follows from the fact that $P_\ord^{\R}(x;1)$ generates $N^{\R}$. The inclusion \eqref{eq:FreeReducCase2} then follows from the fact that $P_\ord^{\R}(x;L)\cap N=P_\ord(x;L)$.
\end{proof}

\bigskip

Note that  ordered progressions are ``almost" symmetric in the sense that for group elements $x_1,\ldots, x_d$ and integers $L_1,\ldots, L_d$  the ordered progression $P_\ord(x,L)$ satisfies 
\begin{equation}\label{eq:inverseProg}
P_\ord(x,L)^{-1}\subset P_\ord(x,L)^d
\end{equation}
(a similar statement holds for real ordered progressions). We deduce the following statement.
\begin{corollary}\label{cor:prop:FreeReducCase}
%Let $N=N_{d,s}$ be the free $s$-nilpotent group on $d$ generators $x=(x_1,\ldots x_d)$, and $L=(L_1,\ldots, L_d)$ be a sequence of positive integers.  
Denote $P=P_\ord(x;L)\cup P_\ord(x;L)^{-1}$ and $P_{\R}=P_\ord^{\R}(x;L)\cup P_\ord^{\R}(x;L)^{-1}$ Then for all $k\in \N$ we have
\[
P^k\subset P_{\R}^k\subset   P^{O_{d,s}(k)}P_{\R}^{O_{d,s}(1)}.
\]
In particular, the inclusion $(N,P)\to (N_{\R},P_{\R})$ is a $(O_{d,s}(1),O_{d,s}(1))$-quasi-isometry.
\end{corollary}
Let us also note the equally trivial fact that for $m\in\N$ we have
\begin{equation}\label{eq:dilateProg}
P_\ord(x,mL)\subset P_\ord(x,L)^{dm}.
\end{equation}

\bigskip

We end this section with an application of Proposition \ref{prop:FreeReducCase}. Breuillard and the second author \cite{bt} have shown that doubling of a Cayley graph at some sufficiently large scale implies uniform doubling at all subsequent scales, as follows.
\begin{theorem}[{\cite[Theorem 1.1]{bt}}]\label{thm:bt}
For every $K\ge1$ there exist $n_0=n_0(K)\in\N$ and $\theta(K) \geq 1$, such that if $S$ is a finite symmetric set inside some group, and if there exists $n\ge n_0$ for which
\begin{equation}\label{doubling}
|S^{2n+1}|\le K|S^n|,
\end{equation}
then for every $m\ge n$ and every $c\in\N$ we have $|S^{cm}|\leq\theta(K)^c|S^m|$.
\end{theorem}
Using the main result of \cite{bgt}, they reduce Theorem \ref{thm:bt} to the following statement.
 
\begin{prop}[{\cite[Theorem 2.8]{bt}}]\label{prop:doublingLargerScale}
Let $K\ge1$ and $s\ge1$. Suppose that $S$ is $K$-approximate subgroup in some $s$-step nilpotent group $G$. Then for every $m\ge1$ the set $S^{m}$ is an $O_{K,s}(1)$-approximate subgroup.
\end{prop}

The proof of Proposition \ref{prop:doublingLargerScale} in \cite{bt} requires some fairly involved computations with nilprogressions. We take the opportunity here to note that Proposition \ref{prop:FreeReducCase} yields an alternative proof of Proposition \ref{prop:doublingLargerScale}. We make use of the following well-known lemma, the proof of which is left as an exercise.

\begin{lemma}\label{lem:metricdoublingQI}
Let $r_0\geq 1$, $j\in \N$ and $\phi:(X,d_X)\to (Y,d_Y)$ be a $(C,K)$-quasi-isometry between two metric spaces, and assume that for all all $r\geq r_0$, all balls of radius $2r$ in $Y$ are covered by at most $j$ balls of radius $r$. Then there exists $j'=O_{j,C,K}(1)$ and $r'_0=O_{j,C,K}(r_0)$ such that for all $r'\geq r'_0$, all balls of radius $2r'$ in $X$ are covered by at most $j'$ balls of radius $r'$.
\end{lemma}

\begin{lemma}\label{lem:guiv.doub}
We keep the notation of Corollary \ref{cor:prop:FreeReducCase}. 
For all $r\geq 1$ there exists $X\subset N_{\R}$ of cardinality $O_{s,d}(1)$ such that
$$P_{\R}^{2r}\subset XP_{\R}^r.$$
\end{lemma} 
\begin{proof}
Applying the automorphism of $N_{\R}$ that maps $x_i^{L_i}$ to $x_i$ reduces the statement to the case where $L_i=1$. In that case, it follows from \cite[Theorem II.1]{Gui} that there exists some $K=O_{d,s}(1)$ such that $|P_{\R}^{2r}|\le K|P_{\R}^r|$ for every $r$, and so the result follows from Lemma \ref{prop:doublingImpliesMetricDoubling}.
\end{proof}
%Write $N=N_{d,s}$ for the free $s$-step nilpotent group on $d$ generators $x_1,\ldots,x_d$, and let $L_1,\ldots,L_d\in\N$. Let $r>1$. Then there exists $X\subset N$ with $|X|\ll_{r,s,d}1$ such that $P_\ord^\R(x;L)^{rm}\subset XP_\ord^\R(x;L)^m$. \Matt{We seem to need $X\subset N$ (see proof of Proposition \ref{prop:doublingLargerScale} below), which was not part of the statement that was written here previously.} \end{lemma}
%\begin{proof} Applying the automorphism of $N_{\R}$ that maps $x_i^{L_i}$ to $x_i$ reduces the statement to the case where $L_i=1$. The statement now simply follows from Guivarc'h's result \Matt{reference?} that nilpotent connected Lie groups have the doubling property. \Matt{Presumably we can take $X\subset N$ by first applying this argument with some smaller $m$ and larger $r$ to get $X\subset N_\R$, and then adjusting $X$ (at the cost of increasing $m$ back to where we wanted it in the first place. Is this easy?}\end{proof}

\begin{proof}[Proof of Proposition \ref{prop:doublingLargerScale}]
By the first paragraph of the proof of \cite[Theorem 2.8]{bt}, it is sufficient to find a positive integer $M_0$ depending only on $K$ and $s$ such that $S^{M_0m}$ is an $O_{K,s}(1)$-approximate subgroup for all $m \ge1$.

We may assume that $S$ generates $G$, in which case it follows from \cite[Theorem 1.5]{nilp.frei} that there exists a normal subgroup $H\lhd G$, as well as $u_1,\ldots,u_d\in G$ and $L_1,\ldots,L_d\in\N$ with $d\le K^{O_s(1)}$ such that
\begin{equation}\label{eq:nilp.Frei}
S\subset H(P_\ord(u;L)\cup P_\ord(u;L)^{-1})\subset S^{K^{O_s(1)}}.
\end{equation}
Write $N=N_{d,s}$ for the free $s$-step nilpotent group on $d$ generators $x_1,\ldots,x_d$, let $\pi:N\to G$ be the homomorphism taking $x_i$ to $u_i$ for every $i$, and write $P=P_\ord(x;L)\cup P_\ord(x;L)^{-1}$. It follows from Corollary \ref{cor:prop:FreeReducCase}, Lemma \ref{lem:metricdoublingQI} and Lemma \ref{lem:guiv.doub} that there exists $r_0\ll_{K,s}1$ such that for every $r\ge r_0$ there exists $X_r\in G$ of cardinality $O_{K,s}(1)$ such that $\pi(P)^{2r}\subset X_r\pi(P)^r$. It also follows from \eqref{eq:nilp.Frei} that there exists $M_0\ll_{K,s}1$ such that $S\subset H\pi(P)^{r_0}\subset S^{M_0}$, and hence
\begin{align*}
S^{2M_0m}&\subset H\pi(P)^{2M_0r_0m}\\
    &\subset X_{mr_0}^{2M_0}H\pi(P)^{r_0m}\\
    &\subset X_{mr_0}^{2M_0}S^{M_0m},
\end{align*}
as required.
\end{proof}

\section{Overview of the general case}\label{sec:overview}
In this section we give an overview of the proofs of Theorems \ref{thm:relative} and \ref{thm:conj2}. The majority of the argument is contained in the following two results, which we prove in the next few sections.
\begin{prop}\label{prop:reducConnected}
Let $(\Gamma_n,S_n)$ and $(m_n)$ be as in Theorem \ref{thm:relative}. Then there exists a natural number $d\ll_D1$ and, for $\omega$-almost every $n$, a simply connected nilpotent Lie group $G_n$ of dimension $d$ and homogeneous dimension at most $D$ such that there is a basis $e^{(n)}$ of the Lie algebra of $G_n$ and a tuple $L^{(n)}\in\N^{d_n}$ such that $(e^{(n)};L^{(n)})$ is in $O_D(1)$-upper-triangular form. There also exists a sequence of integers $j_n\to\infty$ such that, denoting  $Q_n=P^{\R}_\ord(u^{(n)};L^{(n)})$ with $u_i^{(n)}= \exp e_i^{(n)}$, the sequence $(G_n,d_{Q_n}/j_n)$ is relatively compact for the GH topology, $\lim_\omega(G_n,d_{Q_n}/j_n)$ is a connected nilpotent Lie group with a geodesic metric, and
\[
\dim\lim_\omega(\Gamma_n,d_{S_n}/m_n)\leq\dim\lim_\omega(G_n,d_{Q_n}/j_n).
\]
Moreover, under the stronger assumption of Theorem \ref{thm:conj2}, we have in addition
\[
\hdim\lim_\omega(\Gamma_n,d_{S_n}/m_n)\le\hdim\lim_\omega(G_n,d_{Q_n}/j_n).
\]
\end{prop}
%\begin{prop}\label{prop:reducConnected}
%Let $(\Gamma_n,S_n)$ be a sequence of Cayley graphs such that $|S_n^n|\ll_\omega n^D$, and let $(m_n)$ be a sequence of integers such that $m_n\gg_\omega n$. Then there exists a natural number $d\ll_D1$ and, for $\omega$-almost every $n$, a simply connected nilpotent Lie group $G_n$ of dimension $d$ and homogeneous dimension at most $D$ such that there is a basis $e^{(n)}$ of the Lie algebra of $G_n$ and a tuple $L^{(n)}\in\N^{d_n}$ such that $(e^{(n)};L^{(n)})$ is in $O_D(1)$-upper-triangular form. There also exists a sequence of integers $j_n\to\infty$ such that, denoting  $Q_n=P^{\R}_\ord(u^{(n)};L^{(n)})$ with $u_i^{(n)}= \exp e_i^{(n)}$, the sequence $(G_n,d_{Q_n}/j_n)$ is relatively compact for the GH topology, $\lim_\omega(G_n,d_{Q_n}/j_n)$ is a connected nilpotent Lie group with a geodesic metric, and
%\[
%\hdim\lim_\omega(\Gamma_n,d_{S_n}/m_n)\le\hdim\lim_\omega(G_n,d_{Q_n}/j_n).
%\]

\begin{theorem}\label{thm:homdimOfLimit}
Let $G_n$ be a sequence of simply connected nilpotent Lie groups of dimension $d$, and let  $e(=e^{(n)})$ be a basis of the Lie algebra $\g_n$ of $G_n$, and let $(j_n)$ be a sequence of integers tending to $\infty$. Suppose that for each $n$ the convex hull $\Omega_n$ of $\{\pm e^{(n)}_1,\ldots, \pm e^{(n)}_d\}$ satisfies $[\Omega_n,\Omega_n]\subset O_\omega(1)\Omega_n$.
Write $u_i=\exp(e_i)$, and let $Q_n=P_\ord^{\R}(u;1)$. Then $(G_n,d_{Q_n}/j_n)$ is relatively compact for the GH topology, $\lim_\omega(G_n,d_{Q_n}/j_n)$ is a connected nilpotent Lie group with a geodesic metric,
\[
\dim\lim_{\omega} (G_n,d_{Q_n}/j_n)=d,
\]
and if $\hdim(G_n)\le_\omega D$, then
\[
\hdim\lim_{\omega} (G_n,d_{Q_n}/j_n)\le D.
\]
\end{theorem}  

\begin{proof}[Proof of Theorems \ref{thm:relative} and \ref{thm:conj2}]
It follows from Theorem \ref{thm:bft} that the sequence $(\Gamma_n,d_{S_n}/m_n)$ is relatively compact and that every cluster point is a connected nilpotent Lie group with a left-invariant sub-Finsler metric. Given such a cluster point $(G,d)$, on restricting to a subsequence we may assume that $(\Gamma_{n},\frac{d_{S_{n}}}{m_n})\to(G,d)$, and hence by Lemma \ref{lemPrelim:ultra} that $(\Gamma_{n},\frac{d_{S_{n}}}{m_n})\to_\omega(G,d)$. The theorems then follow from Proposition \ref{prop:reducConnected} and Theorem \ref{thm:homdimOfLimit}, noting that upon replacing the pairs $(e^{(n)},L^{(n)})$ coming from Proposition \ref{prop:reducConnected} with $(L^{(n)}e^{(n)},1)$ they indeed satisfy the hypotheses of Theorem \ref{thm:homdimOfLimit}. 
\end{proof}

Proposition \ref{prop:reducConnected} and Theorem \ref{thm:homdimOfLimit} each feature a statement about the relative compactness of a sequence $(G_n,d_{Q_n}/j_n)$ and the form of its ultralimit. Underpinning this aspect of these results is the following statement.

\begin{prop}\label{prop:LimIsAGroup}
Let $(j_n)$ be a sequence of integers going to infinity, and let $(\Gamma_n,P_n)$ be a sequence of Cayley graphs such that either $\Gamma_n$ is a finitely generated nilpotent group of bounded step and $P_n$ is an ordered progression of bounded rank, or $\Gamma_n$ is a connected nilpotent Lie group of bounded step and $P_n$ is a {\it real} ordered progression of bounded rank. Let $d_{P_n}$ denote the word metric associated to $P_n$. Then the sequence $(\Gamma_n,\frac{1}{j_n}d_{P_n})$ is relatively compact for the GH topology and $\lim_{\omega} (\Gamma_n,\frac{1}{j_n}d_{P_n})$ is a connected nilpotent Lie group equipped with a geodesic metric.
\end{prop}

\begin{proof}
We first treat the connected case.
Assume therefore that $P_n=P^{\R}_\ord(u,L)$,  where $u=(u_1,\ldots, u_d)\in \Gamma_n^d$ and $L^{(n)}=(L_1^{(n)},\ldots,L_d^{(n)})\in (1,\infty)^d$, and we assume that $\Gamma_n$ has step at most $s$. Let $N_{\R}$ be the free connected nilpotent Lie group of step $s$ and rank $d$, let $x_1,\ldots x_r$ be elements inducing a basis of $N_{\R}/[N_{\R},N_{\R}]$ and let $\pi_n:N_{\R}\to \Gamma_n$ be the canonical  projection mapping $x_i$ to $u_i$. To prove the relative compactness of $(G_n,P_n)$ it is then sufficient to prove it for $(N_{\R},P^{\R}_\ord(x,L^{(n)}))$. Applying the automorphism of $N_{\R}$ that maps $x_i^{L_i^{(n)}}$ to $x_i$ reduces this statement to the case where $L_i^{(n)}=1$. However, as in the proof of Lemma \ref{lem:guiv.doub}, it follows from \cite[Theorem II.1]{Gui} and the argument of \cite[Lemma 2.2]{bt}
that $(N_{\R},P_\ord^{\R}(x;1))$ has the doubling property, and so its relative compactness is a consequence of Corollary \ref{lemPrelim:compact}.

The ultralimit \[\pi_{\infty}: \lim_{\omega}(N_{\R},d_{P_\ord^{\R}(x;L^{(n)})}/j_n)\simeq N_{\R}\to \lim_{\omega}(\Gamma_n,d_{P_\ord^{\R}(u;L^{(n)})}/j_n)\] of the sequence $\pi_n$ induces a transitive action by isometries of $N_{\R}$ on $\lim_{\omega}(\Gamma_n,d_{P_\ord^{\R}(u;L^{(n)})}/j_n)$.  Therefore, to conclude it is enough to show that the limit is a group, or equivalently that for all sequences $g_n,h_n\in \Gamma_n$ such that $d(g_n,e_n)=O_{\omega}(1)$, and  $d(h_n,e_n)=o_{\omega}(1)$, then $d([g_n,h_n],e_n)=o_{\omega}(1)$. Once again, up to lifting these sequences in $N_{\R}$, we are reduced to proving this statement for sequences in $(N_{\R},P^{\R}_\ord(x,1))$, for which it follows by an easy computation based on Lemma \ref{lem:freeequivalences}. 

The discrete case similarly reduces to the case of $N=N_{r,s}$, the free step $s$ nilpotent group generated by $x_1,\ldots x_r$, equipped with the word metric associated to $P_\ord(x;L)$. Reasoning as in the proof of Proposition \ref{prop:FreeReducCase}, we see that the injection $(N,P_\ord(x;L^{(n)}))\to (N_{\R},P^{\R}_\ord(x,L^{(n)}))$ is a $(O(1),O(1))$-quasi-isometry, so that the proof reduces to the connected case.
 \end{proof}

\section{Reduction to simply connected nilpotent Lie groups}\label{section:reducConnected}
%We start with a useful lemma.
%\begin{lem}\label{lem:properOrdered}
%Let $P=P_\ord(e;L)$ be a proper ordered progression in a torsion-free nilpotent group $\Gamma$. Then 
%$$P_\ord(e;L)^{\R}\cap \Gamma=P_\ord(e;L).$$ 
%\end{lem}
%\begin{proof}

%\end{proof}

The aim of this section is to prove Proposition \ref{prop:reducConnected}. An important tool in the proof is the notion of a \emph{Lie progression}.
\begin{definition}[Lie coset progression]\label{def:Lie.CP}
Let $m\in \N\cup \{\infty\}$, let $\Gamma$ be a group, and let $y_1,\ldots,y_d\in\Gamma$. An ordered coset progression $HP_\ord(y;L)$ is said to be an \emph{$m$-proper Lie coset progression of rank $d$ and homogeneous dimension $D$ in $C$-upper-triangular form} if there exists a connected, simply connected nilpotent Lie group $G$ with Lie algebra $\g$ of homogeneous dimension $D$, with basis $e_1,\ldots,e_d$ such that $(e;L)$ is in $C$-upper-triangular form and $\exp\langle e_1,\ldots,e_d\rangle$ is a subgroup of $G$, and such that, writing $u_i=\exp e_i$ for each $i$, there exists a map $\varphi:\langle u_1,\ldots,u_d\rangle\to\langle y_1,\ldots y_d\rangle$ that is a homomorphism modulo $H$ such that $\varphi(u_i)=y_i$ for each $i$ and such that $P_\ord(u;L)$ is $m$-proper with respect to $\varphi$ modulo $H$. We also define the \emph{injectivity radius} of this Lie coset progression to be the supremum of those $j\in\R$ such that $\varphi$ is injective modulo $H$ on $P_\ord(u;L)^j$.
%If $m\geq 1$ then the group $G$ is unique up to isomorphism \matt{Is this obvious?}, and we define the \emph{homogeneous dimension} of the Lie coset progression $HP_\ord(y_1,\ldots,y_d;L)$ to be the homogeneous dimension of $G$. 
If $H$ is trivial then we say simply that $P_\ord(y;L)$ is an \emph{$m$-proper Lie progression of rank $d$ and homogeneous dimension $D$ in $C$-upper-triangular form}. %When $m=\infty$ and $H=\{1\}$, then $\varphi$ is an isomorphism, and we shall identify $\langle u_1,\ldots,u_d\rangle$ with $\Gamma$. \matt{Do we ever make this identification? If not, we should just delete this last remark.}
\end{definition}

\begin{remark}\label{rem:two.types.of.tri.form}It follows from  \cite[Proposition 4.1]{proper.progs} and that an $m$-proper Lie coset progression of rank $d$ in $C$-upper-triangular form is $m$-proper and in $O_{C,d}(1)$-upper-triangular form as an ordered coset progression.
\end{remark}

\begin{remark}\label{rem:inj/proper}
Adopting the notation from Definition \ref{def:Lie.CP}, it follows from Remark \ref{rem:two.types.of.tri.form} and \cite[Lemma 2.1]{proper.progs} that if $HP$ is an $m$-proper Lie coset progression then it has injectivity radius $\Omega_{C,d}(m^{1/d})$.
\end{remark}

The following result essentially reduces the proof of Proposition \ref{prop:reducConnected} to the study of Lie progressions.

\begin{theorem}\label{thm:ReducTorsionFreeD}
Let $D>0$. Let $(\Gamma_n,S_n)$ be a sequence of Cayley graphs such that
\begin{equation}\label{eq:ReducTorsionFreeD}
|S_n^n|\ll_\omega n^D|S_n|.
\end{equation}
Then there exist integers $t\ll_D1$ and a sequence of positive integers $k_n=o_\omega(n)$ such that for $\omega$-almost every $n$ there exists a Lie coset progression $H_nP_n\subset\Gamma_n$ of rank at most $D$ in $O_D(1)$-upper-triangular form with injectivity radius $\Omega_\omega(n/k_n)$, and a finite subset $X_n\subset S_n^t$ such that $|X_n|\ll_D1$ and such that for every $r\ge1$ we have
\[
X_nH_nP_n^{r}\subset S_n^{rk_n}\subset X_nH_nP_n^{O_D(r)}
\]
Moreover, if $|S_n^n|\ll_\omega n^D$ then $\omega$-almost every $P_n$ has homogeneous dimension at most $D$.
\end{theorem}

We spend the first part of this section proving Theorem \ref{thm:ReducTorsionFreeD}. The starting point is the following result from our first paper.
%We proceed by a recursive argument to obtain the $\Omega(n/k_n)$-properness of $H_nP_n$, from which the bound on the homogeneous dimension will follow. The initial step is taken care of by the following result from our previous paper. 
\begin{prop}[{\cite[Proposition 8.3]{proper.progs}}]\label{prop:inverse}
Let $M,D'>0$, and let $S$ be a finite symmetric generating set for a group $G$ such that $1\in S$. Then there exists $N'=N'_{M,D'}$ such that if $|S^{k}|\le M{k}^{D'}|S|$ for some $k\ge N'$ then there exist integers $t,\eta\ll_{D'}1$, a set $X\subset S^{t}$ such that $|X|\ll_{D'}1$, and a $1$-proper Lie coset progression $HP$ of rank at most $O_{D'}(1)$ in $O_{D'}(1)$-upper-triangular form such that 
\begin{equation}\label{eq:XHPandS}
XHP^r\subset S^{rk}\subset XHP^{\eta r},
\end{equation}
for every $r\in\N$.
\end{prop}
\begin{proof}[Remarks on the proof]
The proposition here does not quite follow from the statement of \cite[Proposition 8.3]{proper.progs}, as that proposition does not specify that the progression of the conclusion should be a Lie progression. To obtain this extra information requires two minor modifications of the proof. The first modification is to replace \cite[Theorem 1.8]{proper.progs} with \cite[Theorem 7.2]{proper.progs} in the proof (\cite[Theorem 7.2]{proper.progs} says explicitly that the progression in the conclusion is a Lie progression, whereas this is suppressed in the statement of \cite[Theorem 1.8]{proper.progs}). The second modification is to remove from the proof of \cite[Proposition 8.3]{proper.progs} the stage where we delete generators that are not necessary for \cite[(8.5)]{proper.progs} to hold.
\end{proof}

The next result allows us to modify the progression coming from Proposition \ref{prop:inverse} to make it proper in some sense.
\begin{prop}\label{prop:dim.reduc.induc}
Let $(\Gamma_n)$ be a sequence of groups, let $\eta\in\N$, and suppose that for $\omega$-almost every $n$ we have a finite symmetric generating set $S_n$ for $\Gamma_n$, a finite subset $X_n$ of $\Gamma_n$, and a Lie coset progression $H_nP_n$ in $\Gamma_n$ of rank $d$ in $C$-upper-triangular form such that
\begin{equation}\label{eq:dim.reduc.induc.assump}
X_nH_nP_n\subset S_n^{k_n r}\subset X_nH_nP_n^{\eta r}
\end{equation}
for every $r\in\N$. Suppose that the injectivity radius of $H_nP_n$ is $o_\omega(n/k_n)$. Then for $\omega$-almost every $n$ there exists $k_n'=o_\omega(n)$ and a Lie coset progression $H_n'P_n'$ of rank strictly less than $d$ and injectivity radius $\Omega_\omega(n/k_n')$ in $O_d(1)$-upper-triangular form such that
\begin{equation}\label{eq:dim.reduc.induc.concl}
X_nH_n'(P_n')^r\subset S_n^{k_n' r}\subset X_nH_n'(P_n')^{O_{C,d}(\eta r)}
\end{equation}
for every $r\in\N$. %\matt{I changed this proposition to use injectivity radius instead of properness, since it's injectivity radius we need in Theorem \ref{thm:ReducTorsionFreeD} and there is a loss involved in using Remark \ref{rem:inj/proper} to pass from properness to injectivity radius.}
\end{prop}

There are two main ingredients in our proof of Proposition \ref{prop:dim.reduc.induc}. The first is one of the key results from our previous paper, as follows.
\begin{prop}[{\cite[Proposition 7.3]{proper.progs}}]\label{prop:reduce.dim.when.not.proper}
Let $H_0P_0$ be a Lie coset progression of rank $d$ in $C$-upper-triangular form that is not $m$-proper. Then there exists an $m$-proper Lie coset progression $HP$ of rank strictly less than $d$ and in $O_d(1)$-upper-triangular form such that
\[
H_0P_0\subset HP\subset H_0P_0^{O_{C,d,m}(1)}.
\]
\end{prop}

The second ingredient is an important property of Lie progressions in upper-triangular form, namely that a power of such a progression is itself roughly equal to such a progression, as follows.

\begin{prop}\label{prop:Powergood}
Let $G$ be a nilpotent Lie group with Lie algebra $\g$ with basis $e_1,\ldots,e_d$ such that $\Lambda=\langle e_1,\ldots,e_d\rangle$ satisfies $[\Lambda,\Lambda]\subset\Lambda$, and let $L_1,\ldots,L_d\in\N$. Then for each $k\in\N$ there exists a basis $e_1',\ldots,e_d'$ for $\g$ and $L'_1,\ldots,L'_d\in\N$ such that $(e',L')$ is in $1$-upper triangular form and such that, writing $u_i=\exp e_i$ and $u_i'=\exp e_i'$ for each $i$, we have
\[
P_\ord(u;L)^k\approx_{d} P_\ord(u',L').
\]
%
%Let $P=P_\ord(e,L)$ be an $\infty$-proper Lie progression of rank $d$ which, such that $(e,L)$ is in $C$-upper triangular form. Then for all $k\in \N$, there exists a basis $e'=(e_1',\ldots,e_d')$, $L'\in \N^d$ in the Lie algebra of $G$, and an $\infty$-proper Lie progression $P'=P_\ord(e',L')$ such that 
%$$\tilde{P}^k\approx_{d,C} P'.$$
%Moreover $(e',L')$ is in $O_{C,d}(1)$-upper triangular form.
\end{prop}

In proving Proposition \ref{prop:Powergood} we make use of the following lemma. We recall from our previous paper that we say a convex body $\Omega$ in a vector space $V$ is \emph{strictly thick} with respect to a lattice $\Lambda$ if there exists some $\lambda<1$ such that $\lambda\Omega\cap\Lambda$ generates $\Lambda$.

\begin{lemma}\label{lem:convexReal}
In the set-up of Proposition \ref{prop:Powergood}, for all $k\in \N$ there exists a symmetric convex body $\Omega\subset\g$, strictly thick with respect to $\langle e_1,\ldots,e_d\rangle$, such that $[\Omega,\Omega]\subset \Omega$ and $P_\ord^{\R}(u;L)^k\approx_d \exp(\Omega)$.
\end{lemma}

\begin{proof}
Note that the step $s$ of $G$ satisfies $s\ll_d1$, so that any bound that in principle depends on $d$ and $s$ may in fact be taken to depend only on $d$. Write $\pi:N^{\R}\to G$ for the homomorphism mapping $x_i$ to $u_i$ for each $i$, write $\pi_\n:\n\to\g$ for the corresponding homomorphism of Lie algebras, write $f_i=\log x_i$ for each $i$, and extend $(f_i)$ to a list of basic commutators as in Section \ref{sec:free.nilp}. It follows from Lemmas \ref{lem:real.B.upper-tri} and \ref{lem:freeequivalences} that $[B_\R(f;(kL)^\chi),B_\R(f;(kL)^\chi)]\subset O_d(1)B_\R(f;(kL)^\chi)$ and $P_\ord^{\R}(x;L)^k\approx\exp(B_\R(f;(kL)^\chi))$. The same argument used in the proof of \cite[Proposition 5.1]{proper.progs} to put the tuple $(e,L)$ into $1$-upper-triangular form then implies that there exists some $L'\in\Z^r$ satisfying $(kL)^{\chi(i)}\le L'_i\le O_d((kL)^{\chi(i)})$ such that
\[
[B_\R(f;L'),B_\R(f;L')]\subset B_\R(f;L'),
\]
and then Lemma \ref{lem:freeequivalences} (iii) implies that $\exp(B_\R(f;L'))\approx_d\exp(B_\R(f;(kL)^\chi))$. The desired conclusion therefore follows by taking $\Omega=\pi_\n(B_\R(f;L'))$.
%
%
%
%
%Note that since $P_\ord^{\R}(u;L)=P_\ord^{\R}(u^L;1)$ it suffices to consider the case where $L=1$.
%We once again consider the projection $\pi:N=N_{d,s}\to \Gamma$ be the projection mapping $x_1,\ldots, x_d$ to $u_1,\ldots, u_d$, and similarly the projection $\pi^{\R}: N^{\R}\to G$. We let $f_i=\log x_i$.
%Let $K=P_\ord^{\R}(x;1).$
%It follows from Lemma \ref{lem:freeequivalences} that $K^k\approx \exp(U),$
%where $U$ is a strictly thick symmetric convex subset satisfying $[U,U]\subset O_{d,s}(1)U$. Up to replacing $U$ by $O_{d,s}(1)U$, we can therefore assume that $[U,U]\subset U$. Then the conclusion follows by taking $\Omega=\pi(U)$. 
\end{proof}

\begin{proof}[Proof of Proposition \ref{prop:Powergood}]
Let $\Omega$ be given by Lemma \ref{lem:convexReal}, so that
\begin{equation}\label{eq:convexReal}
P_\ord^{\R}(u;L)^k\approx_d \exp(\Omega).
\end{equation}
By
%Proposition \ref{prop:nilp.box}
\cite[Proposition 5.1]{proper.progs} there exists a basis $e'$ of the lattice $\langle e_1,\ldots,e_d\rangle \subset\g$, and $L'\in \Z^d$ such that $(e';L')$ is in $1$-upper-triangular form and such that
\[
\Omega\subset B_\R(e';L')\subset O_d(1)\Omega.
\]
The Baker--Campbell--Hausdorff formula then implies that
\begin{equation}\label{eq:Powergood.1}
\exp(\Omega)\approx_d\exp(B_{\R}(e';L')),
\end{equation}
whilst Lemma \ref{lem:freeequivalences} (i) implies that
\begin{equation}\label{eq:Powergood.2}
\exp(B_{\R}(e';L'))\approx_d P_\ord^{\R}(u',L').
\end{equation}
Proposition \ref{prop:Powergood} then follows from \eqref{eq:convexReal}, \eqref{eq:Powergood.1} and \eqref{eq:Powergood.2}.
\end{proof}

\begin{proof}[Proof of Proposition \ref{prop:dim.reduc.induc}]
Since a Lie coset progression of rank $0$ has infinite injectivity radius, by induction on $d$ it suffices to prove the proposition without the conclusion that $H_n'P_n'$ has injectivity radius $\Omega_\omega(n/k_n')$.

%By the definintion of a Lie coset progression, for $\omega$-almost every $n$ we have a connected, simply connected nilpotent Lie group $G_n$ with Lie algebra $\g_n$ with basis $e_1^{(n)},\ldots,e_d^{(n)}$ such that $(e^{(n)};L^{(n)})$ is in $C$-upper-triangular form and $\exp\langle e_1^{(n)},\ldots,e_d^{(n)}\rangle$ is a subgroup of $G_n$, and such that, writing $u_i^{(n)}=\exp e_i^{(n)}$ for each $i$ and $\Lambda_n=\langle u_1^{(n)},\ldots,u_d^{(n)}\rangle$, there exists a map $\varphi_n:\Lambda_n\to\langle y_1^{(n)},\ldots y_d^{(n)}\rangle$ that is a homomorphism modulo $H_n$ such that $\varphi_n(u_i^{(n)})=y_i^{(n)}$ for each $i$, and such that $P_n=P_\ord(y^{(n)};L^{(n)})$. Abbreviate $\tilde{P}_n=P_\ord(u^{(n)};L^{(n)})$.

The fact that $H_nP_n$ has injectivity radius $o_\omega(n/k_n)$ implies in particular that if we define $j_n$ to be the minimum integer such that $\varphi_n^{-1}(H_n)\cap\tilde{P}_n^{j_n}\ne\{e\}$ then
\begin{equation}\label{eq:j<<<n/k}
j_n=o_\omega(n/k_n).
\end{equation}
Proposition \ref{prop:Powergood} and Remark \ref{rem:inj/proper} then imply that there exist Lie coset progressions $H_nQ_n$ of rank $d$ in $O_{C,d}(1)$-upper triangular form such that
\begin{equation}\label{eq:HPHQ}
H_nQ_n\approx_{C,d}H_nP_n^{j_n}
\end{equation}
for $\omega$-almost every $n$, and $\xi>0$ depending only on $C$ and $d$ such that for $\omega$-almost every $n$ the Lie coset progression $H_nQ_n$ is not $\xi$-proper. Proposition \ref{prop:reduce.dim.when.not.proper} therefore implies that there exist Lie coset progressions $H_n'P_n'$ of rank strictly less than $d$ and in $O_d(1)$-upper-triangular form such that
\begin{equation}\label{eq:HP'HQ}
H_n'P_n'\approx_{C,d}H_nQ_n
\end{equation}
for $\omega$-almost every $n$. It then follows from \eqref{eq:dim.reduc.induc.assump}, \eqref{eq:HPHQ} and \eqref{eq:HP'HQ} that \eqref{eq:dim.reduc.induc.concl} holds for some $k_n'\ll_\omega j_nk_n$, which by \eqref{eq:j<<<n/k} is in $o_\omega(n)$, as required.
\end{proof}

\begin{lemma}\label{lem:T>HP}
Let $d,\eta,k,t$ be positive integers. Let $S$ be a finite symmetric generating subset for a group $\Gamma$, let $HP\subset \Gamma$ be a $1$-proper Lie coset progression of rank at most $d$, and let $X\subset S^t$ be such that
\[
XHP^r\subset S^{rk}\subset XHP^{\eta r}
\]
for every $r\in\N$. Then, writing $T=S^{3t}\cap HP^{(d+1)\eta}$, we have 
$$HP\subset (S^{2t+1}\cap HP^{(d+1)\eta})^k\subset T^k.$$
In particular, 
for all $r\geq 1$ we have 
\begin{equation}\label{eq:bounded.dist}
S^{rk}\subset S^tT^{\eta rk}.
\end{equation}
\end{lemma}
\begin{proof}
Let $z\in HP$. Since $HP\subset S^{k}$,  we can write $z=s_k\ldots s_1$, with $s_i\in S$. Let $g_0=1$, and for every $1\leq i\leq k$ write $g_i=s_is_{i-1}\ldots s_1$, observing that therefore $z=g_{k}$. Hence for all $1\leq i\leq k$ there exists $x_i\in X$ and $z_i\in HP^{\eta}$ such that $g_i=x_iz_i$. We therefore have that $z_{i}z_{i-1}^{-1}=x_is_ix_{i-1}^{-1}$, which belongs to $S^{2t+1}\cap HP^{(d+1)\eta}$ by \eqref{eq:inverseProg}.
\end{proof}
\begin{lemma}\label{lem:T<HP}
Under the assumptions of Lemma \ref{lem:T>HP}, and writing $q=\max\{3t\eta,(d+1)\eta\}$, we have $T^k\subset HP^{(d+3)^{|X|}q}$.
\end{lemma}
%\begin{lemma}\label{lem:T<HP}
%Under the assumptions of Lemma \ref{lem:T>HP} we have $T^k\subset HP^{3t(d+2)^{|X|}\eta}$.
%\end{lemma}
\begin{proof}
Since $T^k\subset S^{3tk}$ by definition, we have
\begin{equation}\label{eq:T<P}
T^k\subset XHP^{3t\eta}.
\end{equation}
We first claim that, writing $q=\max\{3t\eta,(d+1)\eta\}$, there exist $X'\subset X$ and $m\in\N$ satisfying $q\le m\le(d+3)^{|X|}q$ such that
\begin{equation}\label{eq:T<P.2}
XHP^{3t\eta}\subset X'HP^m
\end{equation}
and
\begin{equation}\label{eq:T<P.3}
xHP^m\cap HP^{2m}=\varnothing
\end{equation}
for every $x\in X'\backslash\{1\}$.

To prove this claim, we first take $X'=X$ and $m=q$. These choices certainly satsify \eqref{eq:T<P.2}. If they also satisfy \eqref{eq:T<P.3} then the claim is proved. If not, there exists $x\in X'\backslash\{1\}$ such that $xHP^m\cap HP^{2m}\ne\varnothing$, which by \eqref{eq:inverseProg} means that
\[
xHP^m\subset HP^{(d+3)m}.
\]
If we replace $X'$ by $X'\backslash\{x\}$ and $m$ by $(d+3)m$ then \eqref{eq:T<P.2} is therefore still satisfied, and so we check again whether \eqref{eq:T<P.3} is satisfied, and repeat if necessary. This process terminates after at most $|X|$ steps, and so the claim is proved.

We now claim that $T^j\subset HP^m$ for $j=1,\ldots, k$, which of course implies in particular that $T^k\subset HP^{(d+3)^{|X|}q}$, as required. The case $j=1$ follows from the definition of $T$ and the fact that $m\ge q$, so by induction we may assume that $T^{j-1}\subset HP^m$. This then implies that
\begin{align*}
T^j&\subset HP^{m+(d+1)\eta}\cap X'HP^m &\text{(by \eqref{eq:T<P} and \eqref{eq:T<P.2})}\\
    &\subset HP^m     &\text{(by definition of $X'$ and $m$),}
\end{align*}
and the claim is proved.
\end{proof}
%\begin{proof}
%Note that
%\begin{equation}\label{eq:T<P}
%T^k\subset XHP^{3t\eta}.
%\end{equation}
%We run the following algorithm: 
%Assume that for every element $x\in X\setminus\{1\}$, we have $xHP^{5\eta t}\cap HP^{5\eta t}=\varnothing$. Then an easy induction on $j=1,\ldots, k$ shows that 
%$T^k\subset HP^{3t \eta}$. Indeed, this is clear for $j=1$, and assuming that it is true for some $j<k$, we have
%$$T^{j+1}\subset HP^{3t\eta}HP^{2\eta}\subset HP^{5t\eta}.$$
%But combining this with (\ref{eq:T<P}), we have
%$$T^{j+1}\subset HP^{5t\eta}\cap XHP^{3t\eta}=HP^{3t\eta}.$$
%And we are done.
%
%Else, there exists $xHP^{5\eta t}\cap HP^{5\eta t}\neq\varnothing$, which implies that $xHP^{5\eta t}\subset HP^{5(d+2)t\eta}$ (once again using that $P^{-1}\subset P^d$). 
%So we replace $\eta$ by $(d+2)\eta$, and $X$ by $X\setminus\{x\}$.
%This algorithm must terminates after at most $|X|$-steps, meaning that $T^k\subset HP^{(d+2)^{|X|+1}t\eta}$, as required.
%\end{proof}
%Proposition \ref{thm:Previous}  will be used to prove the second item of Theorem \ref{thm:ReducTorsionFreeD}, while the last item will follow from 

%The proof of Theorem \ref{thm:ReducTorsionFreeD} requires us to modify the progression obtained from Proposition \ref{prop:inverse} in order to get  the $\Omega_\omega(n/k_n)$-properness of $H_nP_n$. The bound on the homogeneous dimension will then result from the following lemma.

In order to use the growth bounds in the hypothesis of Theorem \ref{thm:ReducTorsionFreeD}, we use two lower bounds on the growth of sets in nilpotent groups of given dimension. The first of these gives a lower bound on the growth in terms of the homogeneous dimension, as follows.

\begin{lemma}\label{lem:growthLowerBound}
Let $\Lambda$ be a torsion-free nilpotent group of rank $r$ and step $s$, and let $G$ be its Mal'cev completion. Then for every symmetric generating subset $\Sigma$ of $\Lambda$ we have
$$|\Sigma^n|\gg_{s,r} n^{\hdim G}$$
for $n\in\N$.
\end{lemma}
\begin{proof}
Note that $\Sigma$ contains a set of the form $(s_1^{\pm 1},\ldots,s_r^{\pm 1})$, where the $s_i$ generate $\Lambda/[\Lambda,\Lambda]$. Let $q=\dim C^s(\Lambda)\otimes \Q$. As iterated commutators $c_i$ of length $s$ form a generating set of $C^s(\Lambda)\otimes\Q$, we can extract a basis $(c_{i_1},\ldots, c_{i_q})$. Let $c_i=[s_{j_1},\ldots, s_{j_s}]$. For all $(k_1,\ldots, k_s)\in \N^s$, we have 
$$[s_{j_1}^{k_1},\ldots, s_{j_s}^{k_s}]=c_i^{k_1 k_2\ldots k_s}.$$
Hence the ball of radius $n$ contains $\Omega _{r,s}(n^s)$ elements of the form $c_i^{k}$. We deduce that it contains $\Omega_{r,s}(n^{qs})$ elements of $C^s(\Lambda)$. Hence the lemma follows by induction on $s$.
\end{proof}

The second such result gives a lower bound on the \emph{relative} growth in terms of the dimension.
\begin{prop}\label{prop:growth.lower.bound.dim}
Let $d\in \N$, there exists $c=c(d)>0$ such that the following holds.
Let $\Lambda$ be a torsion-free nilpotent group with Mal'cev completion $G$ of dimension $d$. Then for every symmetric generating subset $\Sigma$ of $\Lambda$ we have
\[
|\Sigma^n|\geq cn^d|\Sigma|
\]
for $n\in\N$.
\end{prop}

The proof is, unsurprisingly, by induction on $d$. The induction step is based on the following standard result from additive combinatorcs.
\begin{lemma}\label{lem:sumset.lower.bound}
Let $A$ and $B$ be finite subsets of $\Z$. Then $|A+B|\ge |A|+|B|-1$. In particular, by induction on $n$ we have $|nA|\ge n(|A|-1)$.
\end{lemma}
\begin{proof}
Label the elements of $A$ as $a_1<\ldots<a_p$ and the elements of $B$ as $b_1<\ldots<b_q$, and note that the elements
\[
a_1+b_1,a_2+b_1,\ldots,a_p+b_1,a_p+b_2,\ldots,a_p+b_q
\]
are strictly increasing, and therefore, in particular, distinct.
\end{proof}

\begin{proof}[Proof of Proposition \ref{prop:growth.lower.bound.dim}]
Let $\Lambda'$ be a subgroup of $\Lambda$ such that $\Lambda/\Lambda'\cong\Z$ and such that the Mal'cev completion of $\Lambda'$ has codimension $1$ in $G$, and write $\pi:\Lambda\to\Lambda/\Lambda'$ for the quotient homomorphism. Let $R$ be a set of representatives in $\Lambda$ of the cosets of $\Lambda'$ that have non-trivial intersection with $\Sigma$, and note that by symmetry of $\Sigma$ we have $\Sigma\subset R(\Sigma^2\cap\Lambda')$. Since $|R|=|\pi(\Sigma)|$, this implies that
\begin{equation}\label{eq:sigma.split}
|\Sigma|\le|\pi(\Sigma)||\Sigma^2\cap\Lambda'|.
\end{equation}
On the other hand, given $n\in\N$ let $R_n$ be a set of representatives in $\Lambda$ of the cosets of $\Lambda'$ that have non-trivial intersection with $\Sigma^n$. Since $|R_n|=|\pi(\Sigma)^n|$, Lemma \ref{lem:sumset.lower.bound} implies that $|R_n|\gg n|\pi(\Sigma)|$, and by induction we may assume that $|(\Sigma^2\cap\Lambda')^n|\gg_dn^{d-1}|\Sigma^2\cap\Lambda'|$. Since $R_n(\Sigma^2\cap\Lambda')^n\subset\Sigma^{3n}$, this combines with \eqref{eq:sigma.split} to imply that
\[
|\Sigma^{3n}|\gg_dn^d|\Sigma|,
\]
which implies the required bound.
\end{proof}

\begin{proof}[Proof of Theorem \ref{thm:ReducTorsionFreeD}]
%It follows from \cite[Lemma 8.2]{proper.progs} and \cite[Lemma 2.2]{bt} that there exists $n^{1/4}\leq k\leq n^{1/2}$ for which the ball of radius $k$ is an $O_D(1)$-approximate group.
For $\omega$-almost every $n$ we may apply Proposition \ref{prop:inverse} with $D'=2D$ to obtain integers $t,\eta$ depending only on $D$, and for $\omega$-almost every $n$ a subset $X_n\subset S_n^t$ with $|X_n|\ll_D1$ and a $1$-proper Lie coset progression $H_nP_n$ of rank $O_D(1)$ in $O_D(1)$-upper triangular form such that $X_nH_nP_n^{r}\subset S_n^{r\sqrt{n}}\subset X_nH_nP_n^{\eta r}$ for all $r$. By Lemma \ref{lem:finite.omega.const}, on passing to an appropriate subsequence we may assume that each of these Lie coset progressions has the same rank $d\ll_D1$. By Proposition \ref{prop:dim.reduc.induc}, on replacing $\sqrt{n}$ with some sequence $k_n=o_\omega(n)$ we may assume that $H_nP_n$ has injectivity radius $\Omega_\omega(n/k_n)$. It remains to show that $d\le D$, and that if $|S_n^n|\ll_\omega n^D$ then the homogeneous dimension of $\omega$-almost every $P_n$ is at most $D$.

By the definition of a Lie coset progression, for $\omega$-almost every $n$ we have a connected, simply connected nilpotent Lie group $G_n$ with Lie algebra with basis $e_1^{(n)},\ldots,e_d^{(n)}$ such that $(e^{(n)};L^{(n)})$ is in $C$-upper-triangular form and $\exp\langle e_1^{(n)},\ldots,e_d^{(n)}\rangle$ is a subgroup of $G_n$, and such that, writing $u_i^{(n)}=\exp e_i^{(n)}$ for each $i$ and $\Lambda_n=\langle u_1^{(n)},\ldots,u_d^{(n)}\rangle$, there exists a map $\varphi_n:\Lambda_n\to\langle y_1^{(n)},\ldots y_d^{(n)}\rangle$ that is a homomorphism modulo $H_n$ such that $\varphi_n(u_i^{(n)})=y_i^{(n)}$ for each $i$, such that $P_n=P_\ord(y^{(n)};L^{(n)})$, and such that $\varphi_n$ is injective on $P_\ord(u^{(n)};L^{(n)})^{\Omega_\omega(n/k_n)}$ modulo $H_n$.

Writing $T_n=S_n^{3t}\cap H_nP_n^{(d+1)\eta}$ for every $n$ where this is defined, Lemmas \ref{lem:T>HP} and \ref{lem:T<HP} imply that there exists $\rho\in\N$ depending only on $D$ such that
\begin{equation}\label{eq:rtf.T=HP}
H_nP_n\subset T^{k_n}\subset H_nP_n^{\rho}.
\end{equation}
The fact that the injectivity radius tends to infinity with respect to $\omega$ implies that for $\omega$-almost every $n$ there exists a unique subset $\tilde{T}_n\subset P_\ord(u^{(n)};L^{(n)})^{(d+1)\eta}$ such that $\varphi_n(\tilde{T}_n)H_n=T_nH_n$. The second containment of \eqref{eq:rtf.T=HP} then implies that $\varphi_n$ is injective modulo $H_n$ in restriction to $\tilde{T}_n^{q_n}$ for some sequence $q_n\gg_\omega n$, which implies in particular that
\begin{equation}\label{eq:q.inj.cosets}
|\tilde{T}_n^{q_n}||H_n|=_\omega|T_n^{q_n}H_n|.
\end{equation}
We may assume that $q_n\le_\omega n/4t$, which combined with the fact that $T_n\subset S_n^{3t}$ and $H_n\subset S_n^{o_\omega(n)}$ implies that $T_n^{q_n}H_n\subset_\omega S_n^n$, and hence
\begin{equation}\label{eq:T.upper.bound}
|T_n^{q_n}H_n|\le_\omega|S_n^n|.
\end{equation}
Since $q_n\ge_\omega k_n$, it follows from \eqref{eq:rtf.T=HP} that $\tilde{T}_n^{q_n}$ contains $P_\ord(u^{(n)};L^{(n)})$, and hence generates $\Lambda_n$, and so Lemma \ref{lem:growthLowerBound} implies that
\begin{equation}\label{eq:growth.lb.hdim}
|\tilde{T}_n^{q_n}|\gg_\omega n^{\hdim G_n}
\end{equation}
and Proposition \ref{prop:growth.lower.bound.dim} implies that
\begin{equation}\label{eq:rel/growth.lb.dim}
|\tilde{T}_n^{q_n}|\gg_\omega n^d|\tilde{T}_n|.
\end{equation}
Finally, the fact that $S_n\subset X_nH_nP_n^\eta$ and $X_n\subset S_n^t$ implies that $S_n\subset X_n(S_n^{t+1}\cap H_nP_n^\eta)\subset X_nT_n$, and hence
\begin{equation}\label{eq:T.bigger.than.S}
|T_n|\ge|S_n|/|X_n|\gg_D|S_n|.
\end{equation}
Combining these inequalities gives
\begin{align*}
|S_n^n|&\ge_\omega|T_n^{q_n}H_n|&\text{(by \eqref{eq:T.upper.bound})}\\
   &=_\omega|\tilde{T}_n^{q_n}||H_n|&\text{(by \eqref{eq:q.inj.cosets})}\\
   &\gg_\omega n^d|\tilde{T}_n||H_n|&\text{(by \eqref{eq:rel/growth.lb.dim})}\\
   &\ge_\omega n^d|T_n|&\text{(by definition of $\tilde T_n$)}\\
   &\gg_\omega n^d|S_n|&\text{(by \eqref{eq:T.bigger.than.S})},
\end{align*}
and so \eqref{eq:ReducTorsionFreeD} implies that $d\le_\omega D$, as required. Moreover, \eqref{eq:q.inj.cosets}, \eqref{eq:T.upper.bound} and \eqref{eq:growth.lb.hdim} imply that if $|S_n^n|\ll_\omega n^D$ then $\hdim G_n\le_\omega D$, as required.
\end{proof}

Having reduced the proof of Proposition \ref{prop:reducConnected} to the study of Lie progressions via Theorem \ref{thm:ReducTorsionFreeD}, we now study these progressions a little.

\begin{prop}\label{prop:goodReducReal}
Let $G$ be a connected, simply connected nilpotent Lie group with Lie algebra $\g$ with basis $e_1,\ldots,e_d$, and let $L_1,\ldots,L_d\in\N$ be such that $(e;L)$ is in $C$-upper-triangular form. Write $u_i=\exp e_i$ for each $i$, and write $\Lambda=\langle u_1,\ldots,u_d\rangle$. Then for every $k\in\N$ we have
$$P_\ord(u,L)^k\subset P_\ord^{\R}(u;L)^k\cap\Lambda$$
and
$$P_\ord^{\R}(u;L)^k\cap\Lambda\subset  P_\ord(u,L)^{O_{d,C}(k)}.$$
\end{prop}

\begin{lemma}\label{lem:pp.L2.1}
In the setting of Proposition \ref{prop:goodReducReal} we have $P_\ord^{\R}(u;L)^n\subset P_\ord^{\R}(u;O_{C,d,n}(L))$ for every $n\in\N$.
\end{lemma}
\begin{proof}
The argument of \cite[Proposition 4.1]{proper.progs} implies that $(u;L)$ is in $O_{C,d}(1)$-upper-triangular form over $\R$. The result then follows from using this in place of the upper-triangular form over $\Z$ in the proof of \cite[Lemma 2.1]{proper.progs}.
\end{proof}

\begin{proof}[Proof of Proposition \ref{prop:goodReducReal}]
%Recall that by \cite{bg}, $\exp(\mathfrak{P}(e;L)$, and $P(e;L)$ control each other, and by \cite{bt} the same is true for $P(e;L)$ and  $\overline{P}(e;L)$. On the other hand, $\exp(\mathfrak{P}(e;L)\cap \exp(\mathfrak{P}_{\R}(e;<1)={1}$. Moreover, 
The first conclusion is trivial.
Write $s$ for the nilpotency class of $\Lambda$, noting that $s\ll_d1$, and write $\pi:N=N_{d,s}\to \Lambda$ for the homomorphism mapping $x_i$ to $u_i$ for each $i$. Then note that
\begin{align*}
P_\ord^{\R}(u;L)^k&\subset P_\ord(u;L)^{O_{d}(k)}P_\ord^{\R}(u;L)^{O_d(1)}
                                &\text{(by Proposition \ref{prop:FreeReducCase})}\\
         &\subset P_\ord(u;L)^{O_{d}(k)}P_\ord^{\R}(u;O_{d,C}(L))
                                &\text{(by Lemma \ref{lem:pp.L2.1})},
\end{align*}
and hence that
\begin{align*}
P_\ord^{\R}(u;L)^k\cap\Lambda&\subset P_\ord(u;L)^{O_{d}(k)}P_\ord(u;O_{d,C}(L))\\
           &\subset P_\ord(u;L)^{O_{d,C}(k)} &\text{(by \eqref{eq:dilateProg})},
\end{align*}
which gives the second conclusion.
%
%
%
% 
%Proposition \ref{prop:FreeReducCase} implies that for all $k\in \N$ we have
%$$P_\ord^{\R}(u;L)^k\subset P_\ord(u;L)^{O_{d}(k)}P_\ord^{\R}(u;L)^{O_d(1)}\subset P_\ord(u;L)^{O_{d}(k)}P_\ord^{\R}(u;O_d(1)L).$$
%The desired conclusion then follows from the fact that  $P_\ord^{\R}(u;L)\cap \Lambda=P_\ord(u;L)$
\end{proof}

\begin{corollary}\label{cor:QIreal}
Under the assumptions of Proposition \ref{prop:goodReducReal} the embedding of Cayley graphs $$(\Lambda,P_\ord(u,L)) \hookrightarrow (G,P_\ord^{\R}(u;L))$$ is an $(O_{d,C}(1), O_{d,C}(1))$-quasi-isometry (see Definition \ref{def:qi}).
\end{corollary}

\begin{corollary}\label{cor:limreal}
Let $(G_n)$ be sequence of connected, simply connected nilpotent Lie groups of dimension $d$. For each $n$ write $\g_n$ for the Lie algebra of $G_n$, and let $e^{(n)}=e_1^{(n)},\ldots,e_d^{(n)}$ be a basis for $\g_n$ and $L^{(n)}=(L_1^{(n)},\ldots,L_d^{(n)})\in\N^d$ a $d$-tuple of integers such that $(e^{(n)};L^{(n)})$ is in $C$-upper-triangular form. Write $u_i^{(n)}=\exp e_i^{(n)}$ for each $i$, and write $\Lambda_n=\langle u_1^{(n)},\ldots,u_d^{(n)}\rangle$.  Let $S_n$ be a generating set for $\Lambda_n$, and suppose that $(k_n)$ is a sequence of positive integers such that $k_n=o(n)$ and such that $S_n^{k_n}\approx_{\omega} P_\ord(u^{(n)},L^{(n)})$. Then for every sequence $(m_n)$ of integers such that $m_n\gg n$ the sequences $(\Lambda_{n},d_{S_{n}}/m_n)$ and $\left(G_{n},\frac{d_{P_\ord^{\R}(u^{(n)};L^{(n)})}}{m_n/k_n}\right)$ are precompact and
\[
\lim_\omega\left(\Lambda_{n},\frac{d_{S_{n}}}{m_n}\right)\cong\lim_\omega\left(G_{n},\frac{d_{P_\ord^{\R}(u^{(n)};L^{(n)})}}{m_n/k_n}\right).
\]
%Let $(\Gamma_n,S_n)$ be sequence of Cayley graphs, and let $1\leq k_n=o(n)$ be a  sequence of integers. We assume that
%\begin{itemize}
%\item $\Gamma_n$ is a torsion-free $s$-nilpotent group of rank $d$ of Malcev completion $G_n$;
%\item   $S_n^{k_n}\approx P$ for some  $P$ satisfying the assumptions of Proposition \ref{prop:goodReducReal}.
%\end{itemize}
%\matt{This last condition doesn't make sense - $P$ in Proposition \ref{prop:goodReducReal} lies in a Lie group, whereas $\Gamma_n$ doesn't necessarily embed into a Lie group.}\romain{I try to make it more explicit by mentioning the Malcev completion of $\Gamma_n$. Do you think this is enough?}
%Then, for all sequence $m_n\gg n$, the $\omega$-limits of  $(\Gamma_{n},\frac{d_{S_{n}}}{m_n})$ and $(G_{n},\frac{d_{P_\ord^{\R}(u;L)}}{m_n/k_n})$ are isomorphic.
\end{corollary}
\begin{proof}
Proposition \ref{prop:LimIsAGroup} implies precompactness of the sequence $\left(G_{n},\frac{d_{P_\ord^{\R}(u^{(n)};L^{(n)})}}{m_n/k_n}\right)$, which by Corollary \ref{cor:QIreal} implies that of $(\Lambda_{n},d_{S_{n}}/m_n)$. The isomorphism of the limits then follows from Corollary \ref{cor:QIreal}.
\end{proof}

\begin{lemma}\label{cor:ReducToP}
Let $(\Gamma_n,S_n)$ and $(m_n)$ be as in Theorem \ref{thm:bft}, and write $T_n=S_n^{3t}\cap H_nP_n^{(d+1)\eta}$ for the $\omega$-almost every $n$ for which this is therefore defined by Theorem \ref{thm:ReducTorsionFreeD}. Then the sequence $(\langle H_nP_n\rangle, d_{T_n}/m_n)$ is precompact, and for $\omega$-almost every $n$ there exists a surjective continuous morphism from $\lim_{\omega}(\langle H_nP_n\rangle, d_{T_n}/m_n)$ to 
$\lim_{\omega}(\Gamma_n,d_{S_n}/m_n)$, both being connected nilpotent Lie groups (equipped with left-invariant sub-Fisnler metrics). %\matt{Should we also note in the statement that $\lim_{\omega}(\langle H_nP_n\rangle, d_{T_n}/m_n)$ is a connected nilpotent Lie group with a left-invariant sub-Finsler metric?}
\end{lemma}
\begin{proof}
To see that $(\langle H_nP_n\rangle, d_{T_n}/m_n)$ is precompact, simply note that $|T_n^n|\le|S_n|^{3t}\ll_\omega n^{3tD}$ and apply Theorem \ref{thm:bft}. To prove the existence of the surjective morphism, first note that the inclusion $\psi_n:(\langle H_nP_n\rangle, d_{T_n})\to (\Gamma_n,d_{S_n})$ is $3t$-Lipschitz, whilst Lemma \ref{lem:T>HP}---specifically \eqref{eq:bounded.dist}---implies that for all $r\geq 1$ and $u\in S_n^{rk_n}$ there exists $v\in T_n^{\pm\eta rk_n}$ such that 
$d_{S_n}(u,v)\le t$. The desired morphism is therefore given by Corollary \ref{corultralimitPhi}.
\end{proof}

%\begin{corollary}\label{cor:ReducToP}
%Suppose that the assumptions of Theorem \ref{thm:ReducTorsionFreeD} are satisfied, writing $T_n=S_n^{3t}\cap H_nP_n^{(d+1)\eta}$ for the $\omega$-almost every $n$ for which this is therefore defined, and let $(m_n)$ be a sequence of integers such that $m_n\gg_\omega n$ \matt{is this the correct assumption on $m_n$?}. Then for $\omega$-almost every $n$ the inclusion $\psi_n:(\langle H_nP_n\rangle, d_{T_n})\to (\Gamma_n,d_{S_n})$ is $3t$-Lipschitz, and for all $r\geq 1$ and $u\in S_n^{rk_n}$ there exists $v\in T_n^{\pm\eta rk_n}$ such that 
%$d_{S_n}(u,v)\le t$. \matt{I rewrote this slightly to reflect the proposed change to \eqref{eq:bounded.dist}.} In particular, for every such $n$ there exists a surjective morphism from $\lim_{\omega}(\langle H_nP_n\rangle, d_{T_n}/m_n)$ to 
%$\lim_{\omega}(\Gamma_n,d_{S_n}/m_n)$.
%%and for all $r\geq 1$ and $u\in S_n^{rk_n}$ there exists $v\in T_n^{O_D(rk_n)}$ such that 
%%$d_{S_n}(u,v)\ll_D1$. In particular, for every such $n$ there exists a surjective morphism from $\lim_{\omega}(\langle H_nP_n\rangle, d_{T_n}/m_n)$ to 
%%$\lim_{\omega}(\Gamma_n,d_{S_n}/m_n)$.
%\end{corollary}
%\begin{proof}
%The definition of $T_n$ immediately implies that $\psi_n$ is $3t$-Lipschitz. The existence of $v$ with the required properties follows from Lemma \ref{lem:T>HP}, specifically \eqref{eq:bounded.dist}.
%\end{proof}

\begin{proof}[Proof of Proposition \ref{prop:reducConnected}]
Applying Theorem \ref{thm:ReducTorsionFreeD}, we obtain a sequence
\begin{equation}\label{eq:k=o(n)}
k_n=o(n),
\end{equation}
positive integers $d,t,\eta$ depending only on $D$, and, for $\omega$-almost every $n$, a Lie coset progression $H_nP_n\subset \Gamma_n$ of rank $d$ and injectivity radius $\Omega_\omega(n/k_n)$ in $O_D(1)$-upper-triangular form, and, if $|S_n^n|\ll n_D$, of homogeneous dimension at most $D$. We also obtain a finite subset $X_n\subset S_n^t$ such that $|X_n|\ll_D1$ and such that for every $r\ge1$ we have
\[
X_nH_nP_n^{r}\subset S_n^{rk_n}\subset X_nH_nP_n^{\eta r}.
\]
Defining $T_n=S_n^{3t}\cap H_nP_n^{(d+1)\eta}$, it follows from Lemmas \ref{lem:T>HP} and \ref{lem:T<HP} that
\begin{equation}\label{eq:T=P}
H_nP_n\subset T_n^{k_n}\subset H_nP_n^{O_D(1)},
\end{equation}
and from Lemma \ref{cor:ReducToP} that $(\langle H_nP_n\rangle, d_{T_n}/m_n)$ is precompact and
\begin{equation}\label{eq:reduce.to.X=1.d}
\dim\lim_{\omega}(\Gamma_n,d_{S_n}/m_n)\le\dim\lim_{\omega}(\langle H_nP_n\rangle, d_{T_n}/m_n)
\end{equation}
and
\begin{equation}\label{eq:reduce.to.X=1}
\hdim\lim_{\omega}(\Gamma_n,d_{S_n}/m_n)\le\hdim\lim_{\omega}(\langle H_nP_n\rangle, d_{T_n}/m_n).
\end{equation}
By the definition of a Lie coset progression, for each $n$ for which $P_n$ is defined there exists a connected, simply connected nilpotent Lie group $G_n$---of homogeneous dimension at most $D$ if $|S_n^n|\ll n_D$---with Lie algebra $\g_n$ with basis $e_1^{(n)},\ldots,e_{d}^{(n)}$ and positive integers $L_1^{(n)},\ldots,L_{d}^{(n)}$ such that $(e^{(n)};L^{(n)})$ is in $O_D(1)$-upper-triangular form, such that $\exp\langle e_1^{(n)},\ldots,e_{d}^{(n)}\rangle$ is a subgroup of $G_n$, and such that, writing $u_i^{(n)}=e_i^{(n)}$ and $\Lambda_n=\langle u_1^{(n)},\ldots,u_{d}^{(n)}\rangle$ for each $i$, there exists a map $\varphi_n:\Lambda_n\to\langle y_1^{(n)},\ldots y_{d}^{(n)}\rangle$ that is a homomorphism modulo $H_n$ such that $\varphi_n(u_i^{(n)})=y_i^{(n)}$ for each $i$ and such that $P_n=P_\ord(y^{(n)};L^{(n)})$. Writing $Q_n=P^{\R}_\ord(u^{(n)},L^{(n)})$, we claim that $(G_n,Q_n)$ satisfies the proposition with $j_n=m_n/k_n$.
It follows from \eqref{eq:k=o(n)} that $H_n\subset T_n^{o(n)}$, writing $\Gamma_n'=\langle H_nP_n\rangle/H_n$, and writing $\psi_n:\langle H_nP_n\rangle\to\Gamma_n'$ for the quotient homomorphism, we have
\begin{equation}\label{eq:discard.H}
\lim_{\omega}(\langle H_nP_n\rangle, d_{T_n}/m_n)\cong\lim_{\omega}(\Gamma_n',d_{\psi_n(T_n)}/m_n).
\end{equation}
Since the injectivity radius of $H_nP_n$ tends to infinity with respect to $\omega$, and since $\psi_n(T_n)\subset\psi_n(P_n^{O_D(1)})$, for $\omega$-almost every $n$ there exists a unique subset $\tilde{T}_n\subset P_\ord(u^{(n)};L^{(n)})^{O_D(1)}$ such that $\psi_n\circ\varphi_n(\tilde{T}_n)=\psi_n(T_n)$.
It also follows from \eqref{eq:T=P} and the increasing injectivity radius that for $\omega$-almost every $n$ we have
\[
\tilde{T}_n^{k_n}\approx_DP_\ord(u^{(n)};L^{(n)}),
\]
and so Corollary \ref{cor:limreal} implies that taking $j_n=m_n/k_n$, which tends to infinity by \eqref{eq:k=o(n)}, the sequences $(\Lambda_n,d_{\tilde{T}_n}/m_n)$ and $(G_n, d_{Q_n}/j_n)$ are both precompact and
\begin{equation}\label{eq:limreal}
\lim_{\omega}(\Lambda_n,d_{\tilde{T}_n}/m_n)\cong\lim_{\omega}(G_n, d_{Q_n}/j_n).
\end{equation}
Since $(\Gamma_n',d_{\psi_n(T_n)}/m_n)$ is the image of $(\Lambda_n,d_{\tilde{T}_n}/m_n)$ under $\psi_n\circ\varphi_n$, we have
\begin{equation}\label{eq:Lambda.to.Gamma.d}
\dim\lim_{\omega}(\Gamma_n',d_{\psi_n(T_n)}/m_n)\le\dim\lim_{\omega}(\Lambda_n,d_{\tilde{T}_n}/m_n)
\end{equation}
and
\begin{equation}\label{eq:Lambda.to.Gamma}
\hdim\lim_{\omega}(\Gamma_n',d_{\psi_n(T_n)}/m_n)\le\hdim\lim_{\omega}(\Lambda_n,d_{\tilde{T}_n}/m_n).
\end{equation}
It therefore follows from combining \eqref{eq:reduce.to.X=1.d}, \eqref{eq:reduce.to.X=1}, \eqref{eq:discard.H}, \eqref{eq:limreal}, \eqref{eq:Lambda.to.Gamma.d} and \eqref{eq:Lambda.to.Gamma} that $(G_n,Q_n)$ satisfies the proposition, as claimed.
\end{proof}

We close this section by noting that the proof of Theorem \ref{thm:ReducTorsionFreeD} actually gives a more precise result that we record here for potential future use.

\begin{prop}\label{prop:FiniteJump}
Let $M,D>0$. There exists $N=N_{M,D}\in\N$ such that whenever $n\ge N$ and $(\Gamma,S)$ is a Cayley graph such that $|S^n|\leq Mn^D|S|$, the following holds. %\matt{For consistency with our previous paper, and to avoid confusion with $C$-upper-triangular form, we should make this $C$ into an $M$ at some point.} \romain{Done}
For every $\delta>0$ there exist $j=O_{D,\delta}(1)$, an increasing sequence $n^{\delta}=k_1<k_2<\ldots <k_j<k_{j+1}=n$ such that for every $i=1\ldots j$, there exists
a Lie coset progression $H_iP_i\subset \Gamma$ of rank at most $O_{D,\delta}(1)$ in $O_{D,\delta}(1)$-upper-triangular form, and a finite subset $X_i\subset S^{O_{D,\delta}(1)}$ such that $|X_i|\ll_{D,\delta}1$ and such that
\begin{itemize}
\item for all $r\geq 1$, $X_iH_iP_i^r\subset S_i^{k_ir}\subset XHP^{O_{D,\delta}(r)}$;
\item $H_iP_i$ has injectivity radius $\Omega_{D,\delta}(k_{i+1}/k_i)$.%\matt{Changed this to injectivity radius from properness before.}
%\item $H_iP_i$ is $\Omega_{D,\delta}(k_{i+1}/k_i)$-proper.
%\item under the stronger assumption that  $|S^n|\leq Mn^D$, we can ensure that $P$ has homogeneous dimension $\leq D$.
\end{itemize}
%\matt{Presumably the Lie coset progression is in upper-triangular form as well.}\romain{Added it}
\end{prop}

\section{Reduction to ultralimits of normed Lie algebras}\label{sec:normed}

In this section we give an overview of the proof of Theorem \ref{thm:homdimOfLimit}. The basic strategy is to bound the homogeneous dimension (resp.\ the dimension) of the ultralimit by considering its Lie algebra as the ultralimit of the normed Lie algebras $(\g_n,\|\cdot \|_n)$, where $\|\cdot\|_n$ is a suitable sequence of norms.

There are three main steps. The first step is to define the ultralimit of the $(\g_n,\|\cdot \|_n)$, as follows. Before being a Lie algebra, or even a real vector space, $(\g_n,\|\cdot \|_n)$ is an abelian group equipped with an invariant distance $d_n$ (associated to its norm). Hence we can define $(\g_{\infty},d_{\infty})$ as the ultralimit $\lim_{\omega}(\g_n,d_n)$ as in Definition \ref{def:ultralimitmetricspace}. It is easy to see that $\g_{\infty}$ naturally comes with a real vector space structure, and that $d_{\infty}$ is associated to a norm $\|\cdot\|_{\infty}$. Note moreover that if the dimension of $\g_n$ is in $O_{\omega}(1)$, then the dimension of $\g_{\infty}$ is finite and actually equals $\lim_{\omega}\dim(\g_n)$.  Indeed, up to passing to a subsequence, we can assume that the dimension is fixed, but then all the $(\g_n,\|\cdot \|_n)$ are uniformly bi-Lipschitz equivalent to one another, and hence to $\g_{\infty}$, which therefore has the same dimension. 

 In order for the Lie bracket to {\it converge} as well along the ultrafilter, we need $\|\cdot\|_n$ to be in some way compatible with the Lie bracket. This  compatibility condition turns out to translate geometrically into a ``triangular form'' property for the unit ball, as follows.

\begin{lemma}\label{lem:def.of.ultralim.of.Lie.alg}
Let $(\g_n,\|\cdot \|_n)$ be a sequence of normed Lie algebras of bounded dimension, and suppose that for $\omega$-almost every $n$, the unit ball $B_n=B_{\|\cdot \|_n}(0,1)$ satisfies $[B_n,B_n]\subset O_\omega(1)B_n$. Then the bracket operation on $\g_{\infty}$
\[
[\lim_{\omega}u_n,\lim_{\omega}v_n]=\lim_{\omega}[u_n,v_n]
\]
is well defined and makes $\g_\infty$ into a Lie algebra. Moreover, if the Lie algebras $\g_n$ are nilpotent, then so is $\g_{\infty}$.
\end{lemma}

The next step is to show that the Lie algebras $\g_n$ appearing in Theorem \ref{thm:homdimOfLimit} satisfy the conditions of Lemma \ref{lem:def.of.ultralim.of.Lie.alg}, as follows.

\begin{lemma}\label{lem:normLieAlgebra}
Under the assumptions of Theorem \ref{thm:homdimOfLimit} there exists a sequence of symmetric convex bodies $B_n\subset\g_n$ such that, denoting $\|\cdot\|_n$ the norm whose unit ball is $B_n$,
\begin{itemize}
\item[(i)] $[B_n,B_n]\subset B_n$;
\item[(ii)] for every sequence $g_n\in G_n$ we have $g_n\in Q_n^{O(j_n)}$ if and only if $\|\log g_n\|_{n}=O(1)$;
\item[(iii)] for every sequence $g_n\in G_n$ we have $g_n\in Q_n^{o(j_n)}$ if and only if $\|\log g_n\|_{n}=o(1)$.
\end{itemize}
\end{lemma}

The final step is to bound the homogeneous dimension of the ultralimit Lie algebra, as follows.
\begin{theorem}\label{thmUltraLimitLieAlge}
Let $(\g_n,\|\cdot\|_{B_n})$ be a sequence of normed nilpotent Lie algebras of dimension $d$ such that the unit balls $B_n$ satisfy $[B_n,B_n]\subset O(1)B_n$. Then  
$$\hdim(\g_{\infty})\leq \lim_{\omega} \hdim(\g_n).$$
\end{theorem}

The proof of Theorem \ref{thmUltraLimitLieAlge} is somewhat involved, and so we defer it until the next section. The proofs of Lemmas \ref{lem:def.of.ultralim.of.Lie.alg} and \ref{lem:normLieAlgebra} are more straightforward, and we present them shortly. First, however, let us put all of these results together to prove Theorem \ref{thm:homdimOfLimit}.

\begin{proof}[Proof of Theorem \ref{thm:homdimOfLimit}]
Lemma \ref{lem:normLieAlgebra} (i) implies that $(\g_{\infty},\|\cdot \|_n)$ satisfies the assumptions of Lemma \ref{lem:def.of.ultralim.of.Lie.alg}, making $\g_{\infty}$ a Lie algebra. 
Then Lemma \ref{lem:normLieAlgebra} (ii) and (iii) together with the Baker--Campbell--Hausdorff formula imply that $\exp$ and $\log$ commute with the ultralimits, and so $\g_{\infty}$ canonically identifies with the Lie algebra of $\lim_{\omega} (G_n,d_{Q_n}/j_n)$ (which exists and is a connected nilpotent Lie group equipped with a geodesic metric by Proposition \ref{prop:LimIsAGroup}). The desired bound on the homogeneous dimension therefore follows from Theorem \ref{thmUltraLimitLieAlge} and the fact that $\hdim(G_n)=\hdim(\g_n)$.

Note that the dimension is GH-continuous among normed vector spaces, and therefore that $\dim\lim_{\omega} (G_n,d_{Q_n}/j_n)=d$, as required. Indeed vector spaces of dimension $d$ are $O_d(1)$-bilipschitz equivalent to the euclidean space of dimension $d$. Therefore given a sequence of vector spaces of dimension $d$, its ultralimit is bilipschitz equivalent to the ultralimit of the constant sequence equal to the euclidean space of dimension $d$. 
\end{proof}

\begin{proof}[Proof of Lemma \ref{lem:def.of.ultralim.of.Lie.alg}]
The assumption that $[B_n,B_n]\subset O(1)B_n$ implies that the set of sequences $u_n$ such that $\|u_n\|_{B_n}=_{\omega}o(1)$ is an ideal in the Lie algebra $\{u_n: \|u_n\|_{B_n}=_{\omega}O(1)\}$, and this in turn implies that the Lie bracket is well defined, as claimed. Assume that the Lie algebras $\g_n$ are nilpotent. Since their dimension is bounded, this implies that they are $s$-nilpotent for some $s\in \N$. This clearly implies that $\g_{\infty}$ is $s$-nilpotent as well.
\end{proof}

\begin{proof}[Proof of Lemma \ref{lem:normLieAlgebra}]
Let $N=N_{d,d}$ be the free $d$-step nilpotent Lie group of  rank $d$, let $\n$ be its Lie algebra, let $f=(f_1,\ldots, f_d)$ be a basis of a complement of $[\n,\n]$, and let $x_i=\exp(f_i)$. We consider the projection $\pi_n:N\to G_n$  mapping $x$ to $u$.  As we already noted, $Q_n=P_\ord^\R(u,L)=\pi_n(P_\ord^\R(x,L)$.
 Now the lemma follows from Lemma \ref{lem:freeequivalences}. Indeed, let $\Omega_n=B(f,(j_nL)^{\chi})$, and let $B_n=\pi_n(\Omega_n)$. 
 The first statements is an obvious consequence of Lemma \ref{lem:freeequivalences}, and moreover we deduce that  $\exp(B_n)\approx Q_n^{j_n}$.
The proofs of (ii) and (iii) being similar, we focus on (iii).
  Let $g_n\in G_n$ be such that $g_n\in Q_n^{o(j_n)}$. Let $\tilde{g_n}\in P_\ord^\R(x,L)^{o(j_n)}$ be such that $\pi(\tilde{g}_n)=g_n$. By Lemma \ref{lem:freeequivalences}, $\log \tilde{g}_n\in B(f,(o(j_nL))^{\chi})$. we have that $\log \tilde{g}_n\in \Omega_{o(j_n)}\subset o(1)\Omega_n$. Hence projecting back to  $\g_n$, we see that $\log g_n\in o(1)B_n$, or in other words that $\|\log g_n\|_{B_n}=o(1).$
Conversely, assume that $w_n\in o(1)B_n$, that we lift via $\pi_n$ to an element $\tilde{w_n}\in o(1)\Omega_n$.  Multilinearity of the basic commutators implies that for all $0\leq \eps\leq 1$, we have $\eps B(f,(j_nL)^{\chi})\subset B(f,(\eps^{1/d}j_nL)^{\chi})$. Hence  $\tilde{w_n}\in B(f,(o(j_n)L)^{\chi})$, which by Lemma \ref{lem:freeequivalences}, implies that $\tilde{g}_n:=\exp(\tilde{w}_n)\in P_\ord^\R(x,L)^{o(j_n)}$, from which we deduce that $g_n:=\exp(w_n)\in Q_n^{o(j_n)},$ as required.
\end{proof}

\section{Marked Lie algebras}\label{sec:marked}
In this section we prove Theorem \ref{thmUltraLimitLieAlge}. We fix $s,d\in \N$, let $\n=\n_{s,d}$ be the free $s$-nilpotent Lie algebra of rank $d$, and let $x=(x_1,\ldots,x_d)$ be a basis of a complement of $[\n,\n]$.
The basic objects in consideration in this section are pairs $(\g,\pi)$, where $\g$ is an $s$-nilpotent Lie algebra of rank at most $d$, and $\pi:\n\to \g$ is a surjective morphism. Such a pair is equivalent to a pair $(\g,e)$, where $e$ is a family of $d$ vectors of $\g$ that generate $\g$ as a Lie algebra. We call such a pair a \emph{marked Lie algebra}. An isomorphism of marked Lie algebras $(\g,\pi)\to(\g',\pi')$ is an isomorphism between $\g$ and $\g'$ that commutes with $\pi$ and $\pi'$. Alternatively, an isomorphism of marked Lie algebras $(\g,e)\to(\g',e')$ is an isomorphism between $\g$ and $\g'$ that maps $e$ to $e'$. Isomorphism classes of marked Lie algebras are in one-to-one correspondence with ideals of $\n$.

We define a \emph{relation} in $(\g,e)$ to be a linear combination $\sum_i\lambda_ie_i$ of the basic commutators $e_i$ such that $\sum_i |\lambda_i|=1$, and satisfying $\sum_{i} \lambda_i e_i=0$. 
Alternatively, it is an element of $u\in \ker \pi$ satisfying $\|v\|_{\Omega}=1$ for the norm whose unit ball $\Omega$ is the convex hull of the $\pm e_i$.

We now define a distance on the set of isomorphism classes of marked Lie algebras.

\begin{definition}\label{defn:markedTopology}
Given $\eps>0$ we say that $d((\g,e),(\g',e'))\leq \eps$ if for every relation $u\in \ker \pi$ of norm $1$ there exists a relation $u'\in \ker \pi'$ of norm $1$ such that $\|u-u'\|_{\Omega}\leq \eps$, and for every relation $v'\in \ker \pi'$ of norm $1$ there exists a relation $v\in \ker \pi$ of norm $1$ such that $\|v-v'\|_{\Omega}\leq \eps$.
\end{definition}

\begin{remark}
Observe that this distance coincides with the Hausdorff distance between the compact subsets $\ker \pi\cap \partial \Omega$ and  $\ker \pi'\cap \partial \Omega$, where $\partial \Omega$ is the unit sphere for $\|\cdot\|_{\Omega}$. 
In particular, we deduce that the space of isomorphism classes of marked groups is compact. 
\end{remark}

Let $(\g_n,\pi_n)$ be a sequence of marked Lie algebras of rank at most $d$ and step at most $s$. We now have two natural ways to define an ultralimit of this sequence. First, we may consider the limit $(\bar{\g},\bar{\pi})$ along $\omega$ with respect to the topology defined in Definition \ref{defn:markedTopology}. Note that in this case $\bar\g=\n/\ker\bar\pi$, where
 $\ker\bar\pi$ denotes the subspace of $\n$ spanned by all those $u$ that appear as $\omega$-limits of sequences $u_n$ with $u_n\in\ker\pi_n\cap\partial\Omega$.

On the other hand, we may consider the ultralimit of a sequence of normed Lie algebras as in the previous section. Indeed, since $[\Omega,\Omega]\subset\Omega$ we also have $[\pi(\Omega),\pi(\Omega)]\subset\pi(\Omega)$, and so Lemma \ref{lem:def.of.ultralim.of.Lie.alg} implies that $\g_{\infty}$ is a well-defined Lie algebra. Moreover, the constant sequence $(\n,\|\cdot\|_{\Omega})$ satisfies $\lim_{\omega}(\n,\|\cdot\|_{\Omega})=(\n,\|\cdot\|_{\Omega})$, and it is then easy to verify that $\pi_n$ converges along $\omega$ to some surjective morphism $\pi_{\infty}:\n\to \g_{\infty}$ in the sense that for every bounded sequence $w_n\in (\n,\|\cdot\|_{\Omega})$ we have
$\pi_{\infty}(w_n)=\lim_{\omega}\pi_n(w_n)$.

Conveniently, these two definitions give the same limit, as follows.%\matt{I reorganised this stuff slightly -- if you are happy please just delete this comment -- if not feel free to change it. The old version is commented out below.}

%\begin{definition}
%Let $\omega$ be a non-principal ultrafilter. Then we define the $\omega$-limit $(\bar{\g},\bar{\pi}) =\lim_{\omega}(\g_n,\pi_n)$ of a sequence $(\g_n,\pi_n)$ of marked Lie algebras as
%the limit along $\omega$ with respect to the topology defined in Definition \ref{defn:markedTopology}. Note that $\bar\g=\n/\ker\bar\pi$, where
% $\ker\bar\pi$ denotes the subspace of $\n$ spanned by all those $u$ that appear as $\omega$-limits of sequences $u_n$ with $u_n\in\ker\pi_n\cap\partial\Omega$.
%\end{definition}
%
%The following proposition shows that this notion of limit coincides with the ultralimit of normed Lie algebras considered in the previous section.  First, observe that $[\Omega,\Omega]\subset \Omega$, so that this condition is also satisfied by $\pi(\Omega)$. Therefore, by Lemma \ref{lem:def.of.ultralim.of.Lie.alg}, $\g_{\infty}$ is a well defined Lie algebra. Note that the constant sequence $(\n,\|\cdot\|_{\Omega})$ satisfies that $\lim_{\omega}(\n,\|\cdot\|_{\Omega})=(\n,\|\cdot\|_{\Omega})$, and it is easy to verify that $\pi_n$ converges along $\omega$ to some surjective morphism: $\pi_{\infty}:\n\to \g_{\infty}$, in the sense that for all bounded sequence $w_n\in (\n,\|\cdot\|_{\Omega})$, one has 
%$\pi_{\infty}(w_n)=\lim_{\omega}\pi_n(w_n)$.

\begin{prop}\label{prop:markedlimits/ultralimits}
The ultralimits $(\bar{g},\bar{\pi})$ and $(\g_{\infty},\pi_{\infty})$ are isomorphic.
\end{prop}
\begin{proof}
Let $\sum \lambda_i^{(n)}e_i^{(n)}$ be a sequence of relations in $\g_n$. Then since the $\|e_i^{(n)}\|_{B_n}\leq 1$ and $|\lambda_i^{(n)}|\leq 1$, the limits $\lim_{\omega}e_i^{(n)}= e_i^{\infty}$ and $\lim_{\omega}\lambda_i^{(n)}=\lambda_i^{(\infty)}$ exist, and  we have 
$$0=\lim_{\omega}\sum \lambda_i^{(n)}e_i^{(n)}=\sum \lambda_i^{(\infty)}e_i^{(\infty)},$$
from which we deduce that $\sum \lim_{\omega}\lambda_i^{(n)}e_i^{(n)}$ is a relation in $\g_{\infty}.$

Conversely, if $\sum_i \lambda_ie_i^{(\infty)}$ is a relation in $\g_{\infty}$, then this means that there exists a sequence $u_n$ such that $\|u_n\|_{B_n}=_{\omega}o(1)$ such that 
$\sum_i\lambda_ie_i^{(n)}-u_n=0$.
Note that the condition on $u_n$ means that $u_n=\sum \mu_i^{(n)}e_i^{(n)}$, with $\mu_i^{(n)}=_{\omega}o(1)$. Therefore we have for $\omega$-a.e.\ $n$,
$$\sum_i(\lambda_i -\mu_i^{(n)})e_i^{(n)}=0.$$
Dividing this sum by $\sum_i|\lambda_i-\mu_i^{(n)}|$ (which tends to $1$) shows that $\sum_i \lambda_ie_i^{(\infty)}$ is a limit of relations in $\g_n$.
\end{proof}

Proposition \ref{prop:markedlimits/ultralimits} reduces the proof of Theorem \ref{thmUltraLimitLieAlge} to the following elementary (and probably well-known) fact.

\begin{prop}\label{prop:LowerSemiCont}
The homogeneous dimension of marked Lie algebras is lower semicontinuous with respect to the topology implied by Definition \ref{defn:markedTopology}. 
\end{prop}
We start with two preliminary lemmas.
We let $\g$ be a nilpotent Lie algebra, and write $\g=\CC^1(\g)\subset\CC^2(\g)\subset\ldots$ for the lower central series. For each $u\in \g$ we write $\xi(u)$ for the maximal $k$ such that $u\in\CC^k(\g)$. 
The following lemma trivially follows from the fact that the $\CC^i(\g)$ are closed.
\begin{lemma}\label{lem:xiUpper}
Given a finite-dimensional Lie algebra $\g$, the map $u\to \xi(\pi(u))$ is upper semicontinuous.
\end{lemma}

Write $Z(\g)$ for the centre of $\g$, and for all $k$ write $a_k(\g)=\dim \CC^k(\g)/\CC^{k+1}(\g)$.
\begin{lemma}\label{lem:xi}
Let $u\in Z(\g)$. Then for $k\neq \xi(u)$ we have $a_k(\g/\langle u\rangle) =a_k(\g)$, and 
and for $k= \xi(u)$ we have $a_k(\g/\langle u\rangle) =a_k(\g)-1$. In particular,
$$\hdim(\g/\langle u\rangle)=\hdim(\g)-\xi(u).$$
\end{lemma}

\begin{proof}
These statements are obvious for $k\leq \xi(u)$, and so we may assume that $k>\xi(u)$, and need to prove that $a_k(\g/\langle u\rangle) =a_k(\g)$. Write $p:\g\to \g/\langle u\rangle$. For all $j$ we have $p(\CC^j(\g))=\CC^j(p(\g))$, and so $a_k(p(\g))\leq a_k(\g)$. If this inequality were strict then there would exist $v\in \CC^k(\g)\setminus \CC^{k+1}(\g)$ such that $v+tu\in \CC^{k+1}(\g)$ for some $t\neq 0$. However, this would imply that $tu\in\CC^k(\g)$, and hence $u\in\CC^k(\g)$, contradicting the assumption that $k>\xi(u)$.  
\end{proof}

\begin{proof}[Proof of Proposition \ref{prop:LowerSemiCont}]
First, we claim that given a marked Lie algebra $(\g,\pi)$, we can always find a finite set of generators of $\ker \pi$ (as a vector space) $u_1,\ldots, u_l$ such that for all $1\leq i\leq l$ the element $u_i$ is central modulo the ideal generated by $u_{1},\ldots u_{i-1}$ and such that $\|u_i\|_{\Omega}=1$ ($u_i$ is a relator). To see that, start with an element in $\ker \pi$, and if it is not central in $\n$, take a non-trivial commutator with a unit vector of $\n$, normalise it and repeat this procedure until we get a non-trivial central unit vector in $\ker \pi$. This will give us the first  vector $u_1$. Then do the same replacing $\n$ by $\n/\langle u_1\rangle$, and so on until $(u_1,\ldots, u_l)$ generates the kernel. This proves the claim.

%Now let $(\g_n,\pi_n)$ be a sequence of marked Lie algebras and let $(\bar{g},\bar{\pi})=\lim_{\omega} (\g_n,\pi_n)$. Let $(\bar{u}_1,\ldots, \bar{u}_l)$ be a generating set of $\ker \bar{\pi}$ satisfying the claim. 

For all $n$, we let $(u_1^{(n)}, \ldots, u_{l_n}^{(n)})$ be a generating set of $\ker \pi_n$  satisfying the claim.  Note that since $l_n$ is bounded by the dimension of $\n$, we can assume that it is constant $=l$. Let $\bar{u}_i =\lim_{\omega} u_i^{(n)}$. Clearly, $(\bar{u}_1,\ldots, \bar{u}_l)$ also satisfies the above property. We claim that for all $1\leq i\leq l$, $$\hdim(\n/\langle  \bar{u}_1,\ldots,  \bar{u}_i\rangle)\leq \lim_{\omega} \hdim(\n/\langle  u_1^{(n)},\ldots,  u_i^{(n)}\rangle).$$
The proposition follows by applying this to $i=l$. This statement follows by induction on $i$. The case $i=0$ (corresponding to taking $\n$ all along) is obvious. And if we have proved it for $i$, then one easily sees that 
$$\n/\langle  \bar{u}_1,\ldots,  \bar{u}_{i+1}\rangle =\lim_{\omega} \n/\langle  \bar{u}_1,\ldots,  \bar{u}_{i}, u_{i+1}^{(n)}\rangle.$$
Hence the statement follows by Lemmas \ref{lem:xiUpper} and \ref{lem:xi}, using that $ u_{i+1}^{(n)}$ is central in $ \n/\langle  \bar{u}_1,\ldots,  \bar{u}_{i}\rangle$.
\end{proof}

\begin{proof}[Proof of Theorem \ref{thmUltraLimitLieAlge}]
This follows from Propositions \ref{prop:markedlimits/ultralimits} and \ref{prop:LowerSemiCont}.
\end{proof}

\section{Volume growth and Hausdorff dimension of the limit}\label{sec:Hausdorff}

In this section we examine to what extent the Hausdorff dimension of a cluster point arising from Theorem \ref{thm:relative} can be related to the exponent $D$. We have seen in the introduction that nothing can be said in general, unless the dimension of the limit equals $D$, in which case we shall show that it coincides with the Hausdorff dimension (or equivalently that the limiting metric is Finsler by  \cite[Theorem 13]{B'}).

We start recalling a few well-known facts about Carnot--Carath\'eodory metrics on simply connected nilpotent Lie groups (see for instance \cite{Br}). We often abbreviate the term `Carnot--Carath\'eodory metric' to simply `cc-metric'.

Let $\n$ be the Lie algebra of a simply connected nilpotent Lie group $N$, and let $\m$ be a vector subspace complementary to $[\n,\n]$ equipped with a norm $\|\cdot\|$. 
A smooth path $\gamma: [0,1]\to N$ is said to be horizontal if $\gamma(t)^{-1}\cdot \gamma'(t)$ belongs to  $\m$ for all $t\in [0,1]$. The length of $\gamma$
with respect to $\|\cdot \|$ is then defined by
\begin{equation}\label{eq:length}
l(\gamma)=\int_0^1\|\gamma(t)^{-1}\cdot \gamma'(t)\|dt.
\end{equation}
The Carnot--Carath\'eodory distance associated to $\|\cdot\|$ on $N$ is then defined by  
\begin{equation}\label{eq:distance} 
d(x,y)=\inf_{\gamma} \{l(\gamma): \gamma(0)=x,\; \gamma(1)=y\},
\end{equation}
where the infimum is taken over all piecewise-horizontal paths (i.e.\ concatenations of finitely many horizontal paths).
%It can be shown that since $\m$ generates the Lie algebra $\n$, every pair of points can be joined by a piecewise-horizontal path (see \cite{Gromov} \matt{please can you give a precise reference, e.g.\ section or proposition number?}). 
Note that if $N=\R^m$, so that $\m=\n$, then the Carnot--Carath\'eodory metric is just the distance associated to the norm $\|\cdot\|$ (one easily checks that up to isometry, this distance indeed only depends on the norm $\|\cdot\|$).

 Recall that the real Heisenberg group $\HHH(\R)$ is defined as the matrix group
$$\HHH(\R)=\left\{\left(\begin{array}{cccccc}
1 & u & w \\
0 &  1 & v\\
0 & 0 & 1
\end{array}\right); u,v,w\in \R\right\},$$ and that the discrete Heisenberg  $\HHH(\Z)$ sits inside $\HHH(\R)$ as the cocompact discrete subgroup
consisting of unipotent matrices with integral coefficients. The group $\HHH(\R)$ (resp.\ $\HHH(\Z)$) is 2-step nilpotent; indeed, its centre, which coincides with its derived subgroup, is isomorphic to $\R$ (resp.\ $\Z$), and consists in matrices whose only non-zero coefficient is the top right coefficient. It follows that  $\HHH(\R)/[\HHH(\R),\HHH(\R)]\cong\R^2$ (and similarly $\HHH(\Z)/[\HHH(\Z),\HHH(\Z)]\cong\Z^2$).
The group $\HHH(\R)$ comes with a one-parameter group of automorphisms $(\delta_t)_{t\in \R}$ defined as follows:
\[
\delta_t\left(\left(\begin{array}{cccccc}
1 & u & w \\
0 &  1 & v\\
0 & 0 & 1
\end{array}\right)\right)=\left(\begin{array}{cccccc}
1 & tu & t^2w \\
0 &  1 & tv\\
0 & 0 & 1
\end{array}\right).
\]

Given a norm $\|\cdot \|$ on $\R^2$, there exists a unique left-invariant Carnot--Carath\'eodory metric $d_{cc}$ on $\HHH(\R)$ that projects to $\|\cdot\|$ and is scaled by $\delta_t$, i.e.\ such that $d_{cc}(e,\delta_t(g))=td_{cc}(e,g)$ for all $t\in\R^*_+$ and all $g\in \HHH(\R)$.  Normalise the Haar measure on $\HHH(\R)$ so that the ball of radius $1$ has volume $1$. It follows from the formula defining $\delta_t$ that  for all $r>0$, the ball of radius $r$ for this metric has measure equal to $r^4$. In particular, the Hausdorff dimension of $\HHH(\R)$ equals $4$ (see \cite[Lemma C.3]{semmes}).

For every $j,k,l\in \N$ write
\[
P(j,k)=\left(\begin{smallmatrix}1 & [-j,j]\cap\Z & [-k,k]\cap\Z\\0&1&[-l,l]\cap\Z\\0&0&1\end{smallmatrix}\right)
\quad\text{and}\quad
P_\R(j,k)=\left(\begin{smallmatrix}1 & [-j,j] & [-k,k]\\0&1&[-l,l]\\0&0&1\end{smallmatrix}\right),
\]
and write $S(j,k)=P(j,k)\cup P(j,k)^{-1}\subset \HHH(\Z)$
and $S_{\R}(j,k)=P_{\R}(j,k)\cup P_{\R}^{-1}(j,k)\subset \HHH(\R)$.  A straightforward calculation shows that $P^{-1}(j,k)\subset P(j,j^2+k)$, and so if $k\geq j^2$ we have
\begin{equation}\label{eq:PtoS}
P(j,k)\subset S(j,k)\subset P(j,2k),
\end{equation}
and similarly 
\[
P_{\R}(j,k)\subset S_{\R}(j,k)\subset P_{\R}(j,2k).
\]
The following lemma is an easy consequence of a celebrated result of Pansu \cite{Pansu} (strictly speaking, we use \cite[Theorem 1.4 ]{Br} which is a slight generalisation due to Breuillard). 
\begin{lemma}\label{lem:squeezing}
Let $j_n$ and $q_n$ be sequences of positive integers such that $q_n\to \infty$, and let $S_n=S(j_n,j_n^2)$. Then $(\HHH(\Z), d_{S_n}/q_n)$ GH-converges to $\HHH(\R)$ endowed with the cc-metric associated to the $\ell^{\infty}$-norm on $\R^2.$ 
\end{lemma}
\begin{proof}
We consider a relatively compact subset $D\subset \HHH(\R)$ such that $\HHH(\R)$ is the disjoint union of left $\HHH(\Z)$-translates of $D$.  
For every left-invariant distance $d$ on $\HHH(\Z)$, define a $\HHH(\Z)$-invariant pseudo distance $d_{\R}$ on $\HHH(\R)$, defined as $d^{\R}(g,h)=d(\gamma,\lambda)$ where $g\in \Gamma D$ and $h\in \lambda D$. Note that the embedding $(\HHH(\Z),d)\to (\HHH(\R),d^{\R})$ is an isometry (in particular it is essentially surjective, in the sense that every point of $\HHH(\R)$ lies at distance zero from a point of $\HHH(\Z)$).
For every $j\in \N$,
we denote  
\[P_{j}=\left(\begin{smallmatrix}1 & \pm j & 0\\ 0 & 1 &\pm j \\0&0&1\end{smallmatrix}\right),\]
and $S_{j}=P_{j}\cup P_j^{-1}$.

Note that we have 
\[d^{\R}_{S(j_n,j_n^2)}\leq d^{\R}_{S(j_n,j_n^2)}\leq d_{S_{\R}(j_n,j_n^2)}.\]
 Applying the automorphism $\delta_{j_n^{-1}}$ to $\HHH(\R)$, we obtain 
\[d^{\R}_{S(1,1)}\leq d^{\R}_{S(j_n,j_n^2)}(\delta_{j_n^{-1}}(\cdot),\delta_{j_n^{-1}}(\cdot))\leq d_{S_{\R}(1,1)}.\]

But then it follows from \cite[Theorem 1.4]{Br} that 
$(\HHH(\R),d^{\R}_{S(j_n,j_n^2)}(\delta_{j_n^{-1}}(\cdot),\delta_{j_n^{-1}}(\cdot))/q_n)$ GH-converges  to $\HHH(\R)$ endowed with the cc-metric associated to the $\ell^{\infty}$-norm on $\R^2$. Now  $\delta_{j_n}$ defines an isometric isomorphism from $(\HHH(\R),d^{\R}_{S(j_n,j_n^2)}(\delta_{j_n^{-1}}(\cdot),\delta_{j_n^{-1}}(\cdot))/q_n)$ to 
$(\HHH(\R),d^{\R}_{S(j_n,j_n^2)}/q_n)$. And we have seen at the beginning of this proof that the injection $(\HHH(\Z),d_{S(j_n,j_n^2)}/q_n)\to (\HHH(\R),d^{\R}_{S(j_n,j_n^2)}/q_n)$ is an isometry as well. So we deduce that $(\HHH(\Z),d_{S(j_n,j_n^2)}/q_n)$ converges to $\HHH(\R)$ endowed with the cc-metric associated to the $\ell^{\infty}$-norm on $\R^2$, and 
the lemma follows. \end{proof}

\begin{proof}[Proof of Proposition \ref{prop:HausdorffNotCV}]
Pick a sequence $a_n$ such that $n^{a_n}\approx f(n)$. 
Let $S_n=S(n,n^{3-a_n})$.
Using \eqref{eq:PtoS}, one easily checks that $|S_n^n| \asymp n^8$ and $|S_n|\asymp n^{5-a_n}$, and so $|S_n^n|/|S_n|\ll n^{3} n^{a_n}\asymp f(n) n^3$.

We now prove that the limit is isometric to the cc-metric associated to the $\ell^{\infty}$-norm. This follows from Lemma \ref{lem:squeezing} together with the fact that there exists $t_n$ going to zero such that for all $t_n n\leq k\leq n$
\[S(n,n^{3-a_n})^{k} = S(n,n^2)^{k}.\]
Indeed, a direct calculation (left to the reader), shows that for all positive integers $i,j,m$ such that $(mi)^2\geq 10mj$, 
\[S(i,j)^m= S(i,i^2)^m.\]
Therefore, in order to find $t_n$, one has to solve in $t$ the equation
$t^4n^8= 10t^3n^{8-a_n}$,
giving $t=10n^{-a_n}.$
\end{proof}

We now move on to the proof of Proposition \ref{prop:LimitFinsler}, starting with an immediate consequence of Proposition \ref{prop:growth.lower.bound.dim}. 
\begin{lemma}\label{lem:lowergrowth}
Let $\Lambda$ be a torsion-free nilpotent group with Mal'cev completion $G$ of dimension $d$. Let $\Sigma$ be a symmetric finite generating set of $\Lambda$, let $n\geq 0$ and $C$ be such that 
\[
|\Sigma^n|\leq Cn^d|\Sigma|.
\]
Then for every $i,j$ with $0<i<j\leq n$ we have
\[
c(j/i)^d|\Sigma^i|\leq |\Sigma^j|\leq (C/c^2)(j/i)^d|\Sigma^i|,
\]
where $c$ is the constant coming from Proposition \ref{prop:growth.lower.bound.dim}.
\end{lemma}

\begin{proof}[Proof of Proposition \ref{prop:LimitFinsler}]
We first reduce to the case where $\Gamma_n$ has no torsion and its Mal'cev completion has dimension $D$.
Indeed, combining Theorem \ref{thm:ReducTorsionFreeD}
and Lemma \ref{cor:ReducToP}, we are reduced to the case where $\Gamma_n$ is a torsion-free nilpotent group. In that case, the connected Lie group $G_n$ of Proposition \ref{prop:reducConnected} is simply its Mal'cev completion (this is Corollary \ref{cor:limreal}). Finally we deduce from Theorem \ref{thm:homdimOfLimit} that the dimension of $G_n$ equals the dimension of the ultralimit, i.e.\ $D$.   

It then follows from Lemma \ref{lem:lowergrowth} that in the ultralimit $G_{\infty}$ (with $m_n=n$) we have
\[
(c^2/C) r^D|B(e,1)|\leq |B(e, r)| \leq (1/c)r^D|B(e,1)|
\]
for every $r\leq 1$. This is well known to imply that the Hausdorff dimension equals $D$  (see \cite[Lemma C.3]{semmes}). Finally we deduce from \cite[Theorem 13]{B'} that the metric is Finsler. 
\end{proof}

     \appendix

\section{Detailed growth of nilprogressions}
The purpose of this appendix is to prove Theorem \ref{thm:tao}. We essentially reproduce Tao's original proof of the theorem in \cite[\S4]{tao}, but substitute in various results from the present paper and its predecessor \cite{proper.progs} to make the argument finitary.

Proposition \ref{prop:inverse} essentially reduces Theorem \ref{thm:tao} to the study of the growth of Lie progressions. The key result on that topic is the following.
\begin{prop}\label{prop:tao.prog.version}
Let $P$ be a Lie progression of rank $d$ in $C$-upper-triangular form for some $C$. Then there exists a non-decreasing continuous piecewise-linear function $f:[0,\infty)\to[0,\infty)$ with $f(0)=0$ and at most $O_d(1)$ distinct linear pieces, each with a slope that is a natural number at most $O_d(1)$, such that
\[
\log|P^m|=\log|P|+f(\log m)+O_d(1)
\]
for every $m\in\N$.
\end{prop}
\begin{remark}In fact, we do not need the Lie progression $P$ appearing in Proposition  \ref{prop:tao.prog.version} to be in upper-triangular form; it would be enough for the lattice $\Lambda$ generated in the Lie algebra by the basis defining $P$ to satisfy $[\Lambda,\Lambda]\subset\Lambda$ as in Proposition \ref{prop:Powergood}.
\end{remark}

We prove Proposition \ref{prop:tao.prog.version} are using material from Section \ref{section:reducConnected} and the following results, which are almost identical in spirit to part of Tao's argument in \cite[\S4]{tao}.

\begin{prop}\label{prop:prog.growth.poly}
Let $G$ be a connected, simply connected nilpotent Lie group with Lie algebra $\g$ with basis $e_1,\ldots,e_d$, and write $\Lambda=\langle e_1,\ldots,e_d\rangle$. Suppose that $\exp\Lambda$ is a subgroup of $G$. Write $u_i=\exp e_i$ for each $i$, let $L_1,\ldots,L_d$ be positive integers, and write $P=P_\ord(u;L)$. Then there exists a polynomial $f$ of degree $O_d(1)$ with positive coefficients such that
\[
|P^m|\asymp_d f(m).
\]
\end{prop}

\begin{lemma}\label{lem:poly=piecewise.mono}
Let $f$ be a polynomial of degree $k$ with no negative coefficients. Then there exists a continuous, piecewise-monomial function $h$ of increasing degree with at most $k+1$ pieces, each of which has a positive coefficient, such that
\[
f(x)\asymp_kh(x)
\]
for $x>0$.
\end{lemma}
\begin{proof}
Following Tao \cite[\S4]{tao}, write $f(x)=\alpha_0+\alpha_1x+\ldots+\alpha_kx^k$, and note that for every $a,b>0$ we have $a+b\asymp\max\{a,b\}$, so that
\[
f(x)\asymp_k\max_i\alpha_ix^i.
\]
We may therefore take $h(x)=\max_i\alpha_ix^i$.
\end{proof}

Before we prove Proposition \ref{prop:prog.growth.poly}, let us see how it combines with Lemma \ref{lem:poly=piecewise.mono} to imply Proposition \ref{prop:tao.prog.version}.

\begin{proof}[Proof of Proposition \ref{prop:tao.prog.version}]
We proceed by induction on $d$, taking the trivial case $d=0$ as the base. If $P$ is infinitely proper then the result follows from Proposition \ref{prop:prog.growth.poly} and Lemma \ref{lem:poly=piecewise.mono}. If not then $P$ has finite injectivity radius, say $k\in\N$, and Proposition \ref{prop:prog.growth.poly} and Lemma \ref{lem:poly=piecewise.mono} imply that there exists a non-decreasing continuous piecewise-linear function $f:[0,\infty)\to[0,\infty)$ with $f(0)=0$ and at most $O_d(1)$ distinct linear pieces, each with a slope that is a natural number at most $O_d(1)$, such that
\begin{equation}\label{eq:m<k}
\log|P^m|=\log|P|+f(\log m)+O_d(1)
\end{equation}
for every $m\le k$. Let $c>0$ be a constant to be determined by depending only on $d$. Proposition \ref{prop:Powergood} implies that there exists a Lie progression $P_0$ of rank $d$ in $1$-upper-triangular form and with injectivity radius $O_d(1)$ such that
\begin{equation}\label{eq:P^k->P}
P^{ck}\subset P_0\subset P^{O_d(ck)}.
\end{equation}
Remark \ref{rem:inj/proper} implies that there exists $q\ll_d1$ such that $P_0$ is not $q$-proper. Proposition \ref{prop:reduce.dim.when.not.proper} therefore implies that there exists a Lie coset progression $HP_1$ of rank strictly less than $d$ in $O_d(1)$-upper-triangular form such that $P_0\subset HP_1\subset P_0^{O_d(1)}$, and then \eqref{eq:P^k->P} implies that $P^{ck}\subset HP_1\subset P^{O_d(ck)}$. We may thus choose $c$ so that
\begin{equation}\label{eq:P->P'}
HP_1^r\subset P^{rk}\subset HP_1^{O_d(r)}
\end{equation}
for every $r\in\N$.

Following Tao \cite[\S4]{tao}, write $\hat P$ for the image of $P_1$ in the quotient $\langle HP_1\rangle/H$ and note that
\begin{equation}\label{eq:deal.with.H}
|HP_1^r|=|H||\hat P^r|
\end{equation}
for every $r\in\N$. By induction we may assume that there exists a non-decreasing continuous piecewise-linear function $h:[0,\infty)\to[0,\infty)$ with $h(0)=0$ and at most $O_d(1)$ distinct linear pieces, each with a slope that is a natural number at most $O_d(1)$, such that
\[
\log|\hat P^r|=\log|\hat P|+h(\log r)+O_d(1)
\]
for every $r\in\N$. This implies in particular that $|\hat P^{\ell r}|\ll_{\ell,r}|\hat P^r|$ for every $\ell,r\in\N$, and hence, combined with \eqref{eq:P->P'} and \eqref{eq:deal.with.H}, that $|P^{rk}|\asymp_d|H||\hat P^r|$ for every $r\in\N$. It follows that
\[
\log|P^{rk}|=\log|P^k|+h(\log r)+O_d(1)
\]
for every $r\in\N$. The monotonicity of $\log|P^m|$ in $m$ therefore implies that
\[
\log|P^k|+h(\log\textstyle\lfloor\frac{m}{k}\rfloor)-O_d(1)\le\log|P^m|\le\log|P^k|+h(\log\textstyle\lceil\frac{m}{k}\rceil)+O_d(1)
\]
for every $m\ge k$, and so the fact that the slope of $h$ is bounded in terms of $d$ implies that $\log|P^m|=\log|P^k|+h(\log\textstyle{\frac{m}{k}})+O_d(1)$ for every $m\ge k$. In light of \eqref{eq:m<k}, this implies that
\[
\log|P^m|=\log|P|+f(\log k)+h(\log m-\log k)+O_d(1)
\]
for every $m\ge k$. Since $h(0)=0$, this combines with \eqref{eq:m<k} to give the desired result for all $m$.
\end{proof}

We start our proof of Proposition \ref{prop:prog.growth.poly} with a general lemma from linear algebra, which was essentially present in \cite[\S4]{tao}.
\begin{lemma}\label{lem:cramer}
Let $r\ge d$, $x_1,\ldots x_r$ to span $\R^d$, and $M_1,\ldots,M_r>0$. There exist $i_1<\ldots<i_d$ such that
\[
B_\R(x;M)\subset r\cdot B_\R(x_{i_1},\ldots,x_{i_d};M_{i_1},\ldots,M_{i_d}).
\]
\end{lemma}
\begin{proof}
Pick the $i_1<\ldots<i_d$ that maximise
\[
\vol(B_\R(x_{i_1},\ldots,x_{i_d};M_{i_1},\ldots,M_{i_d}))
\]
(where the volume is defined with respect to the canonical basis of $\R^d$). 
On reordering the $x_i$'s, we can assume that $i_j=j$ for all $1\leq j\leq d$. 
It suffices to show that for a given $x_k$ we have
\begin{equation}\label{eq:one.basis.vector}
M_kx_k\in  B_\R(x_{1},\ldots,x_{d};M_{1},\ldots,M_{d}).
\end{equation}
View each $x_i$ as a column vector with coordinates with respect to the basis $x_1,\ldots,x_d$, and write $A$ for the $d\times d$ matrix with columns $M_{1}x_{1}\ldots,M_{d}x_{d}$. Cramer's rule implies that the solution $y\in\R^d$ to the equation $Ay=M_kx_k$ satisfies
\[
|y_j|=\frac{\vol(B_\R(x_{1},\ldots,x_{j-1},x_k,x_{j+1},\ldots,x_{d};M_{1},\ldots,M_{j-1},M_k,M_{j+1},\ldots,M_{d}))}{\vol(B_\R(x_{1},\ldots,x_{d};M_{1},\ldots,M_{d}))}
\]
(i.e.\ with $x_{j}$ replaced by $x_k$ in the numerator). By maximality, this implies in particular that $|y_j|<1$, which gives \eqref{eq:one.basis.vector}, as required.
\end{proof}

From now on in this section we adopt the notation of Proposition \ref{prop:prog.growth.poly}. We also extend $e=(e_1,\ldots,e_d)$ to the list $\overline{e}=(e_1,\ldots,e_r)$ of basic commutators, as defined in Section \ref{sec:free.nilp}. Given a measurable subset $B$ of $\g$ we write $\vol(B)$ for the measure of $B$, normalised so that the lattice $\Lambda$ has determinant $1$.

\begin{prop}\label{prop:prog.as.nilbox}
For every $m\in\N$ we have
\[
|P^m|\asymp_d\vol(B_\R(\overline e;(mL)^\chi)).
\]
\end{prop}
\begin{proof}
A result of van der Corput \cite{vdC} states that for any convex body $K\subset\R^d$ we have $|\Z^d\cap K|\ge\frac{1}{2^d}\vol(K)$, which in this case implies that
\[
\vol(B_\R(\overline e;(mL)^\chi))\le2^d|\Lambda\cap B_\R(\overline e;(mL)^\chi)|
\]
for every $m\in\N$. Since
\begin{align*}
\exp\big(\Lambda\cap B_\R(\overline e;(mL)^\chi)\big)
    &\subset\exp\Lambda\cap P_\ord^\R(u;L)^m&\text{(by Lemma \ref{lem:freeequivalences} (i))}\\
    &\subset P^{O_d(m)}&\text{(by Proposition \ref{prop:FreeReducCase}),}
\end{align*}
this implies that for some constant $B_d>0$ we have
\begin{equation}\label{eq:vol<}
\vol(B_\R(\overline e;(mL)^\chi))\ll_d|P^{B_dm}|
\end{equation}
for every $m\in\N$. Since
\[
\vol(B_\R(\overline e;(\alpha mL)^\chi))\ll\alpha^{\sum_i|\chi(i)|}\vol(B_\R(\overline e;(mL)^\chi))
\]
for every $\alpha\ge1$ and every $m$, on replacing $m$ by $B_d^{-1}m$ in \eqref{eq:vol<} we obtain
\[
\vol(B_\R(\overline e;(mL)^\chi))\ll_d|P^m|,
\]
which is one direction of what we need to prove.

For the other direction, write $Q=[0,1]^d$ and note that
\begin{align*}
P^m&\subset\exp\Lambda\cap P_\ord^\R(u;L)^m&\text{(since $\exp\Lambda$ is a group)}\\
   &\subset\exp\big(\Lambda\cap B_\R(\overline e;(O_d(mL))^\chi)\big)&\text{(by Lemma \ref{lem:freeequivalences} (i))},
\end{align*}
and hence
\begin{align*}
|P^m|&\le|\Lambda\cap B_\R(\overline e;(O_d(mL))^\chi)|\\
    &=\vol\big(\big(\Lambda\cap B_\R(\overline e;(O_d(mL))^\chi)\big)+Q\big)\\
    &\le\vol B_\R(\overline e;(O_d(mL))^\chi)&\text{(since $Q\subset B_\Z(e;L)$)}\\
    &\ll_d\vol B_\R(\overline e;(mL)^\chi),
\end{align*}
as required.
\end{proof}

\begin{lemma}
There exists a polynomial $f$ of degree $O_d(1)$ with no negative coefficients such that
\[
\vol(B_\R(\overline e;(mL)^\chi))\asymp_df(m).
\]
\end{lemma}
\begin{proof}
We continue to follow Tao \cite[\S4]{tao}. Set
\[
f(m)=\sum_{1\le i_1<\ldots<i_d\le r}\vol(B_\R(e_{i_1},\ldots,e_{i_d};(mL)^\chi)),
\]
and note that
\[
f(m)=\sum_{1\le i_1<\ldots<i_d\le r}\vol(B_\R(e_{i_1},\ldots,e_{i_d};L^\chi))m^{\sum_{j=1}^d|\chi(i_j)|},
\]
which is certainly of the required form. It follows from Lemma \ref{lem:cramer} that
\[
\vol(B_\R(\overline e;(mL)^\chi))\le r^df(m),
\]
whilst the fact that $B_\R(e_{i_1},\ldots,e_{i_d};(mL)^\chi)\subset B_\R(\overline e;(mL)^\chi)$ for every $i_1,\ldots,i_d$ implies that
\[
f(m)\le\textstyle{r\choose d}\vol(B_\R(\overline e;(mL)^\chi)).
\]
\end{proof}

\begin{proof}[Proof of Theorem \ref{thm:tao}]
It follows from Proposition \ref{prop:inverse} that there exist $|X|\ll_D1$, and a Lie coset progression $HP$ of rank at most $O_D(1)$ in $O_D(1)$-upper-triangular form such that 
\begin{equation}\label{eq:tao.red.to.prog}
HP^m\subset S^{mn}\subset XHP^{O_D(m)},
\end{equation}
for every $m\in\N$. Write $\hat P$ for the image of $P$ in the quotient $\langle HP\rangle/H$, and note that $|HP^m|=|H||\hat P^m|$ for every $m\in\N$. Since Proposition \ref{prop:tao.prog.version} implies that $|\hat P^{rm}|\ll_{r,D}|\hat P^m|$ for every $r,m\in\N$, this combines with \eqref{eq:tao.red.to.prog} to imply that
\[
|S^{mn}|\asymp_D|H||\hat P^m|
\]
for every $m\in\N$, and so the theorem follows from Proposition \ref{prop:tao.prog.version}.
\end{proof}

\begin{remark}\label{rem:bgt.location}
The ineffectiveness in Theorem \ref{thm:tao} arises from the single use of Proposition \ref{prop:inverse} at the beginning of the proof; the ineffectiveness of Proposition \ref{prop:inverse} in turn arises from a single use of \cite[Theorem 1.6]{bgt} (see also \cite[Remark 8.6]{proper.progs}).
\end{remark}


\begin{thebibliography}{99}

\bibitem{bass}
H. Bass, The degree of polynomial growth of finitely generated nilpotent groups, \textit{Proc. London
Math. Soc.} (3) \textbf{25} (1972), 603--614.
%\bibitem{bilu}
%Y. Bilu. Structure of sets with small sumset, Structure theory of set addition, \textit{Ast\'{e}risque} \textbf{258} (1999), 77--108.
\bibitem{bft}
I. Benjamini, H. Finucane and R. Tessera. On the scaling limit of finite vertex transitive graphs with large diameter.  \textit{Combinatorica} \textbf{36} (2016), 1--41.
\bibitem{B} V. N. Berestovski\'i. Homogeneous manifolds with an intrinsic metric. I, \textit{Sibirsk. Mat. Zh.} \textbf{29}(6) (1988), 17--29.
\bibitem{B'} V. N. Berestovski\'i. Locally Compact Homogeneous Spaces with Inner Metric,  \textit{J. Gen. Lie Theory Appl.} \textbf{9}(1) (2015), 6 pp. 
\bibitem{Br} E. Breuillard. Geometry of locally compact groups of polynomial growth and shape of large balls. arXiv:0704.0095.
%\bibitem{eb.minnesota}
%E. Breuillard. Approximate groups and Hilbert's fifth problem, \textit{Recent Trends in Combinatorics}, \textit{The IMA Volumes in Mathematics and its Applications} \textbf{159} (2016).
\bibitem{bg}
E. Breuillard and B. J. Green. Approximate groups. I. The torsion-free nilpotent case, \textit{J. Inst. Math. Jussieu} \textbf{10}(1) (2011), 37--57.
\bibitem{sol.lin}
E. Breuillard and B. J. Green. Approximate groups, II: the solvable linear case, \textit{Q. J. Math.} \textbf{62}(3) (2011), 513--521.
%\bibitem{unitary}
%E. Breuillard and B. J. Green. Approximate groups, III: the unitary case, \textit{Turk. J. Math.} \textbf{36} (2012), 199--215.
%\bibitem{bgt.lin}
%E. Breuillard, B. J. Green and T. C. Tao. Approximate subgroups of linear groups, \textit{Geom. Funct. Anal.} \textbf{21}(4) (2011), 774--819.
%\bibitem{bgt.lin.note}
%E. Breuillard, B. J. Green and T. C. Tao. A note on approximate subgroups of $GL_n(\C)$ and uniformly nonamenable groups, arXiv:1101.2552.
\bibitem{bgt}
E. Breuillard, B. J. Green and T. C. Tao. The structure of approximate groups, \textit{Publ. Math. IHES.} \textbf{116}(1) (2012), 115--221.
%\bibitem{nfdl}
%E. Breuillard, B. J. Green and T. C. Tao. A nilpotent Freiman dimension lemma, \textit{European J. Combin.} \textbf{34}(8) (2013), 1287--1292.
%%%%%%%%%%%\bibitem{bgt.survey} E. Breuillard, B. J. Green and T. C. Tao. Small doubling in groups, \textit{Proc. of the Erd\H{o}s centenery conference} (2013); arXiv:1301.7718.
\bibitem{bt}
E. Breuillard and M. C. H. Tointon. Nilprogressions and groups with moderate growth, \textit{Adv. Math.} \textbf{289} (2016), 1008--1055.
\bibitem{BH} M. Bridson and A. Hefliger, \textit{Metric Spaces of Non-Positive Curvature}, Grundl. der Math. Wiss. \textbf{319}, Springer Verlag (1999).
%\bibitem{cassels}
%J.W.S. Cassels. \textit{An introduction to the geometry of numbers}, Springer (1959).
\bibitem{Com}
W. W. Comfort and S. Negrepontis. \textit{The theory of ultrafilters}, Berlin, New York: Springer-Verlag
(1974).
%\bibitem{cor-gre}
%L. J. Corwin and F. P. Greenleaf. Representations of nilpotent Lie groups and their applications. Part 1: Basic theory and examples, \textit{Cambridge studies in advanced mathematics} \textbf{18}, Cambridge Univ. Press (1990).
%\bibitem{fkp}
%D. Fisher, N. H. Katz and I. Peng. Approximate multiplicative groups in nilpotent Lie groups, \textit{Proc. Amer. Math. Soc.} \textbf{138}(5) (2010), 1575--1580.
%\bibitem{freiman}
%G. A. Freiman. Foundations of a structural theory of set addition, \textit{Translations of Mathematical Monographs} \textbf{37}, Amer. Math. Soc., Providence, RI (1973). Translated from the 1966 Russian version, published by Kazan Gos. Ped. Inst..
\bibitem{gill-helf}
N. Gill and H. A. Helfgott. Growth in solvable subgroups of $GL_r(\Z/p\Z)$, \textit{Math. Ann.} \textbf{360}(1) (2014) 157--208.
%%%%%%%%%%%%%%%%\bibitem{app.grps} B. J. Green. Approximate groups and their applications: work of Bourgain, Gamburd, Helfgott and Sarnak, \textit{Current events bulletin of the AMS} (2010), arXiv:0911.3354.
%%%%%%%%%%%%%%%\bibitem{ben.icm} B. J. Green. Approximate algebraic structure, to appear in \textit{Proc. ICM 2014}; arXiv:1404.0093.
\bibitem{FY} K. Fukaya and T. Yamaguchi.  The fundamental groups of almost nonnegatively curved manifolds,  \textit{Math. Ann.} \textbf{136} (1992), 253--333.
%\bibitem{green-ruzsa} B. J. Green and I. Z. Ruzsa. Freiman's theorem in an arbitrary abelian group, \textit{J. Lond. Math. Soc.} \textbf{75}(1) (2007), 163--175.
\bibitem{gromov}
M. Gromov. Groups of polynomial growth and expanding maps, \textit{Publ. Math. IHES} \textbf{53} (1981), 53--73.
\bibitem{Gromov} M. Gromov. \textit{Metric structures for Riemannian and non-Riemannian spaces}, volume 152 of Progress in Mathematics. Birkh\"auser Boston Inc., Boston, MA (1999). Based on the 1981 French original,
With appendices by M. Katz, P. Pansu and S. Semmes.
\bibitem{Gui} Y. Guivarc'h. Croissance polynomiale et p\'eriodes des fonctions
harmoniques, \textit{Bull. Sc. Math. France} \textbf{101} (1973), 333--379.
\bibitem{hall}
M. Hall. \textit{The theory of groups}, Amer. Math. Soc./Chelsea, Providence, RI (1999).
\bibitem{Hall2}
M. Hall. A basis for free Lie rings and higher commutators in free groups, \textit{Proc. Amer. Math. Soc.} \textbf{1} (1950), 575--581.
%\bibitem{hei} J. Heinonen. \textit{Lectures on Analysis on Metric Spaces}, Springer, New York, NY, 2001.
\bibitem{Pansu} P. Pansu.
Croissance des boules et des g\' eod\' esiques ferm\' ees dans les nilvari\' et\' es.
\textit{Ergodic Theory Dyn. Syst.} \textbf{3} (1983), 415--445.
%\bibitem{helfgott1}
%H. A. Helfgott. Growth and generation in $SL_2(\Z/p\Z)$, \textit{Ann. Math.} \textbf{167} (2008), 601--623.
%\bibitem{helfgott2}
%H. A. Helfgott. Growth in $SL_3(\Z/p\Z)$, \textit{J. Eur. Math. Soc.} \textbf{13}(3) (2011), 761--851.
%%%%%%%%%%%%%%%%\bibitem{helf.survey} H. A. Helfgott. Growth in groups: ideas and perspectives, arXiv:1303.0239.
%\bibitem{hrush}
%E. Hrushovski. Stable group theory and approximate subgroups, \textit{J. Amer. Math. Soc.} \textbf{25}(1) (2012), 189--243.
%\bibitem{pyb-sza}
%L. Pyber and E. Szab\'o. Growth in finite simple groups of Lie type of bounded rank, arXiv:1005.1858.
%\bibitem{pyb-sza.2}
%L. Pyber and E. Szab\'o. Growth in linear groups, \textit{Thin groups and superstrong approximation}, MSRI Publications \textbf{61} (2013).

\bibitem{Roe}
J. Roe, Lectures on Coarse Geometry, Univ. Lecture Ser., vol. 31, American Mathematical Society, Providence, RI, 2003
%\bibitem{ruzsa.Z}
%I. Z. Ruzsa. Generalized arithmetical progressions and sumsets, \textit{Acta Math. Hungar.} \textbf{65}(4) (1994), 379-388.
%\bibitem{ruzsa}
%I. Z. Ruzsa. An analog of Freiman's theorem in groups, Structure theory of set addition, \textit{Ast\'{e}risque} \textbf{258} (1999), 323--326.
%%%%%%%%%%%%%%%%%%%\bibitem{sand.survey} T. Sanders. The structure theory of set addition revisited, \textit{Bull. Amer. Math. Soc.} \textbf{50} (2013), 93--127.
\bibitem{semmes} S. Semmes, Finding curves on general spaces through quantitative topology, with applications
to Sobolev and Poincar\'e inequalities, Selecta Math. 2 (1996), 155--295.
\bibitem{st} Y. Shalom and T. Tao, A finitary version of Gromov's polynomial growth theorem. \textit{Geom. Funct. Anal.} \textbf{20} (2010), no. 6, 1502--1547.
%%%%%%%%%%%%%%%%%%%\bibitem{tao.product.set}
%%%%%%%%%%%%%%%%%%%T. C. Tao. Product set estimates for non-commutative groups, \textit{Combinatorica} \textbf{28}(5) (2008), 547--594.
%\bibitem{tao.solv}
%T. C. Tao. Freiman's theorem for solvable groups, \textit{Contrib. Discrete Math.} \textbf{5}(2) (2010), 137--184.
\bibitem{tao}
T. C. Tao. Inverse theorems for sets and measures of polynomial growth, \textit{Q. J. Math.} \textbf{68}(1) (2017), 13--57.
\bibitem{proper.progs}
R. Tessera and M. C. H. Tointon. Properness of nilprogressions and the persistence of polynomial growth of given degree, arXiv:1612.05152.
\bibitem{nilp.frei}
M. C. H. Tointon. Freiman's theorem in an arbitrary nilpotent group, \textit{Proc. London Math. Soc.} (3) \textbf{109} (2014), 318--352.
\bibitem{resid}
M. C. H. Tointon. Approximate subgroups of residually nilpotent groups, arXiv:1509.03876.



\bibitem{vdD-W}
L. van den Dries and A .J. Wilkie. Gromov's theorem on groups of polynomial growth and elementary logic, \textit{J. Alg.} \textbf{89} (1984), 349--374.
\bibitem{vdC}
J. G. van der Corput. Verallgemeinerung einer Mordellschen Beweismethode in der Geometrie der Zahlen II, \textit{Acta Arith.} \textbf{2} (1936), 145--146.
\end{thebibliography}
\end{document}